\documentclass[12pt]{amsart}

\usepackage{amsthm}
\usepackage[usenames,dvipsnames]{xcolor}
\usepackage{graphicx}
\usepackage{stmaryrd} %has \mapsfrom, \llbracket, \rrbracket
\usepackage{tikz}
\usepackage{tikz-cd}
\usepackage{ bbold }
\usetikzlibrary{arrows}
\usepackage{verbatim}
\usepackage{amssymb,amsfonts}
\usepackage[all,arc]{xy}
\usepackage{enumerate}
\usepackage{mathrsfs, mathabx}
\usepackage{xcolor}[2007/01/21]
\usepackage{hyperref}
\usepackage{soul}
\usepackage{esint} % \varointctrclockwise
\usepackage{multirow} % \multirow (for a simple table)
\usepackage[all]{xy}
\usepackage[margin=1.0in]{geometry}
\usepackage[
   open,
   openlevel=2,
   atend
 ]{bookmark}[2011/12/02]
\usepackage{graphics}
\usepackage{enumitem,kantlipsum}

\def\Y{{\mathcal Y}}

\def\S{{\mathcal S}}

\def\A{{\mathcal A}}
\def\B{{\mathcal B}}

\def\E{{\mathbb E}}

\renewcommand{\phi}{\varphi}

\renewcommand{\P}{\mathbb{P}}
\newcommand{\cC}{\mathcal{C}}
\newcommand{\Hs}{\mathcal{H}}

\makeatletter
\let\@@pmod\pmod
\DeclareRobustCommand{\pmod}{\@ifstar\@pmods\@@pmod}
\def\@pmods#1{\mkern4mu({\operator@font mod}\mkern 6mu#1)}
\makeatother

\def\Z{{\mathbb Z}}
\def\N{{\mathbb N}}

\def\Q{{\mathbb Q}}

\def\SS{{\mathcal S}}

\def\Var{{\operatorname{Var}}}

\def\Ap{{\operatorname{Ap}}}
\def\KV{{\operatorname{KV}}}

\newtheorem{question}{Question}
\newtheorem{thm}{Theorem}

\newtheorem{cor}[thm]{Corollary}
\newtheorem{lemma}[thm]{Lemma}
\newtheorem*{defn}{Definition}

\newtheorem{proposition}[thm]{Proposition}
\newtheorem{corollary}[thm]{Corollary}

\makeatletter
\@namedef{subjclassname@2020}{%
  \textup{2020} Mathematics Subject Classification}
\makeatother

\begin{document}

\title[The Expected Embedding Dimension of a Numerical Semigroup]{The Expected Embedding Dimension, type and weight of a Numerical Semigroup}

\author{Nathan Kaplan}
\address{Department of Mathematics, University of California, Irvine,
340 Rowland Hall, Irvine, CA 92697}
\email{nckaplan@math.uci.edu}

\author{Deepesh Singhal}
\address{Department of Mathematics, University of California, Irvine, 340 Rowland Hall, Irvine, CA 92697}
\email{dsinghal@uci.edu}

\date{November 14th, 2022}



\begin{abstract}
    We study statistical properties of numerical semigroups of genus $g$ as $g$ goes to infinity.  More specifically, we answer a question of Eliahou by showing that as $g$ goes to infinity, the proportion of numerical semigroups of genus $g$ with embedding dimension close to $g/\sqrt{5}$ approaches $1$.  We prove similar results for the type and weight of a numerical semigroup of genus $g$.
\end{abstract}

\maketitle

\section{Introduction}

Let $\N_0$ denote the monoid of nonnegative integers. A \emph{numerical semigroup} $S$ is a submonoid of $\N_0$ for which $|\N_0\setminus S|<\infty$. The \emph{set of gaps} of $S$ is $\Hs(S)=\N_0\setminus S$. The \emph{Frobenius number} of $S$, denoted by $F(S)$, is the largest element of $\Hs(S)$.  The \emph{genus} of $S$, denoted by $g(S)$, is the number of elements of $\Hs(S)$. The \emph{multiplicity} of $S$, denoted $m(S)$, is the smallest nonzero element of $S$. For a general reference on numerical semigroups, see \cite{Gen ref 1}.

There has been extensive recent interest in counting numerical semigroups ordered by genus and in studying invariants of `typical' numerical semigroups of given genus.  Our goal is to prove several statistical results about these numerical semigroups.  Every numerical semigroup has a unique minimal generating set, which we denote by $\A(S)$. This set consists of positive elements of $S$ that are not the sum of two positive elements of $S$.  That is, $\A(S)=(S\setminus\{0\})\setminus((S\setminus\{0\})+(S\setminus\{0\}))$. The size of the minimal generating set is called the \emph{embedding dimension} of $S$, and is denoted by $e(S)=|\A(S)|$. 
The \emph{pseudo-Frobenius} numbers of $S$ are defined as follows:
\[
PF(S)=\{P\in\Hs(S)\colon \text{for every }s\in S\setminus\{0\}\text{ we have }P+s\in S\}.
\]
The number of pseudo-Frobenius numbers is the \emph{type} of $S$, denoted by $t(S)=|PF(S)|$.
The \emph{weight} of $S$ is defined as
\[
w(S)=\bigg(\sum_{x\in\Hs(S)}x\bigg)-\frac{g(g+1)}{2}.
\]
The motivation for studying the weight of a numerical semigroup comes from the theory of Weierstrass semigroups of algebraic curves.  For a reference, see \cite[Chapter 1, Appendix E]{ACGH}.

There are infinitely many numerical semigroups, so in order to prove statistical statements about their invariants we must order them in some way.  Let $\SS_g$ denote the set of numerical semigroups of genus $g$.  It is not difficult to show that $\SS_g$ is finite. There is an extensive literature about how the size of this set varies with $g$. Let $N(g)=|\SS_g|$ be the number of numerical semigroups of genus $g$.  Let $\phi=\frac{1+\sqrt{5}}{2}$ be the golden ratio.
\begin{thm}\cite[Theorem 1]{Zhai}\label{thm_zhai}
There exists a constant $S> 3.78$ such that
\[
\lim_{g\to\infty}\frac{N(g)}{\phi^g} = S.
\]
\end{thm}
We denote the uniform probability distribution on $\SS_g$ by $\P_g$. If $X$ is a random variable on $\SS_g$, we denote its expectation by $\E_{g}[X]$ and its variance by $\Var_{g}[X]$.

\begin{question}
How are the quantities $m(S), F(S), e(S), t(S)$, and $w(S)$ distributed as we vary through the semigroups in $\SS_g$?
\end{question}

Let $\gamma=\frac{5+\sqrt{5}}{10}=\frac{1}{\sqrt{5}}\phi$.  Kaplan and Ye show that most numerical semigroups $S \in \SS_g$ have multiplicity close to $\gamma g$ and Frobenius number close to twice the multiplicity \cite{KaplanYe}. The main goal of this paper is to prove analogous statements for $e(S), t(S)$, and $w(S)$.

\begin{thm}\cite[Proposition 16 and Theorem 4]{KaplanYe}\label{Kaplan Mult}
For fixed $\epsilon > 0$, we have
\begin{enumerate}
\item
\[
\lim_{g\to\infty}\P_{g}[|m(S)-\gamma g |< \epsilon g]=1,\ \text{and}
\]
\item 
\[
\lim_{g\to\infty}\P_{g}[|F(S)-2m(S)|< \epsilon g]=1.
\]
\end{enumerate}
\end{thm}

Theorem \ref{Kaplan Mult} implies that for fixed $\epsilon>0$,
\[
\lim_{g\to\infty}\P_{g}[|F(S)-2\gamma g |< \epsilon g]=1.
\]
Singhal strengthens Theorem \ref{Kaplan Mult}(2) in \cite{Singhal}.
\begin{proposition}\cite[Theorem 8]{Singhal}\label{F 2m}
Given $\epsilon>0$, there is an $M(\epsilon) >0$ such that for all $g > 0$ we have
\[
\P_g[|F(S)-2m(S)|>M(\epsilon)]<\epsilon.
\]
\end{proposition}
\noindent Recently, Zhu has proven a stronger result of this kind but we will not need it for the applications in this paper \cite[Theorem 6.1]{Zhu}.

\subsection{The Distribution of Invariants of Numerical Semigroups in $\SS_g$}

We show that most numerical semigroups of genus $g$ have embedding dimension close to $\frac{1}{\sqrt{5}}g$, type close to $(1-\gamma)g$, and weight close to $\frac{1}{10\varphi} g^2$.

\begin{thm}\label{main result}
Fix $\epsilon > 0$. We have
\begin{enumerate}
\item 
\[
\lim_{g\to\infty}\P_{g}\left[\left|e(S)-\frac{1}{\sqrt{5}} g \right|< \epsilon g\right]  =  1,\ \text{and}
\]
\item 
\[
\lim_{g\to\infty}\P_{g}\left[\left|t(S)-(1-\gamma) g \right|< \epsilon g\right]  =  1.
\]
\end{enumerate}
\end{thm}

As a direct consequence we can compute the expected values of $e(S)$ and $t(S)$ taken over semigroups in $\SS_g$.
\begin{corollary}\label{main_cor}
We have
\begin{enumerate}
\item
\[
\lim_{g\to\infty}\frac{1}{g}\E_g[e(S)]  =  \frac{1}{\sqrt{5}},\ \text{and}
\]
\item 
\[
\lim_{g\to\infty}\frac{1}{g}\E_g[t(S)]  =  1-\gamma.
\]
\end{enumerate}
\end{corollary}

In \cite[Theorem 22]{KaplanYe}, Kaplan and Ye use the Hardy-Ramanujan asymptotic formula for the number of partitions of $n$ to prove that with high probability, a random $S \in \SS_g$ satisfies 
\[
0.03519 g^2 < w(S) < 0.0885 g^2.
\]
They state that it would be interesting to try to improve the constants in these inequalities.  We achieve this goal by proving a result for the distribution of $w(S)$ taken over semigroups in $\SS_g$ analogous to Theorem \ref{main result}.  The proof strategy for this result is different than the strategy for Theorem \ref{main result}.  We first compute the expected value and variance of $w(S)$, and then use these results to deduce our result about the distribution.
\begin{thm}\label{Weight result}
Fix $\epsilon > 0$. We have
\begin{enumerate}
    \item \[
\lim_{g\to\infty}\frac{1}{g^2}\E_g[w(S)] =  \frac{1}{10\phi},\ \text{and}
\]
    \item \[
\lim_{g\to\infty}\P_{g}\left[\left|w(S)-\frac{1}{10\phi} g^2 \right|< \epsilon g^2\right]  = 1.
\]
\end{enumerate}
\end{thm}
\noindent Note that $\frac{1}{10\varphi} \approx 0.0618$.

The main tool to prove this theorem is a result about the independence of the probability that a set of elements is contained in a semigroup in $\SS_g$.  We expect that this result is of independent interest. In joint work with Bras-Amor\'os we have applied it to a different statistical problem about semigroups in $\SS_g$ \cite{Shape NS}.  We return to this result in Section \ref{sec:ind_prob}.

Wilf asked in \cite{Wilf} whether every numerical semigroup $S$ satisfies
\[
\frac{1}{e(S)}\leq \frac{F(S)+1-g(S)}{F(S)+1}.
\]
This question is now commonly known as \emph{Wilf's conjecture} and has been the subject of extensive work in the numerical semigroups community.  Sammartano proved that if $S$ if a numerical semigroup with $e(S) \ge m(S)/2$, then $S$ satisfies Wilf's conjecture \cite[Theorem 18]{Sammartano}.  This result was improved by Eliahou, who showed that if $S$ satisfies $e(S) \ge m(S)/3$ then $S$ satisfies Wilf's conjecture \cite[Theorem 1]{Eliahou}.  Eliahou explains that Delgado has observed that more than $99.999\%$ of the numerical semigroups with genus at most $45$ have $e(S) \ge m(S)/3$. Eliahou asks whether the proportion of such semigroups goes to $1$ as $g$ goes to infinity. A direct consequence of Theorem \ref{main result}(1) is that not only do most semigroups satisfy the condition in Eliahou's result but most also satisfy the stronger condition in Sammartano's result.
\begin{corollary}
We have
\[
\lim_{g\to\infty}\P_g[e(S)\geq m(S)/2]= 1.
\]
\end{corollary}

\begin{figure}
    \centering
    \includegraphics[height=0.2\textheight]{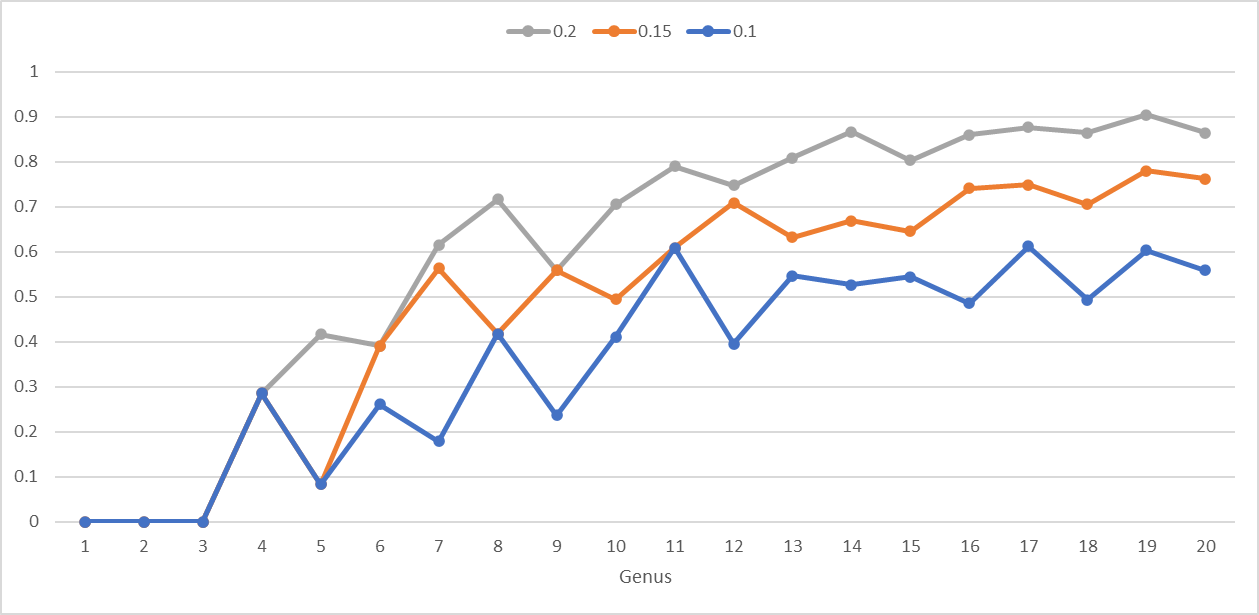}
    \caption{Proportion of $S \in \SS_g$ with $|e(S)-\frac{1}{\sqrt{5}}g|<\epsilon g$. \newline
    Plotted for $\epsilon=0.2$, $\epsilon=0.15$ and $\epsilon=0.1$.}
    \label{fig:my_label}
\end{figure}

\begin{figure}
    \centering
    \includegraphics[height=0.2\textheight]{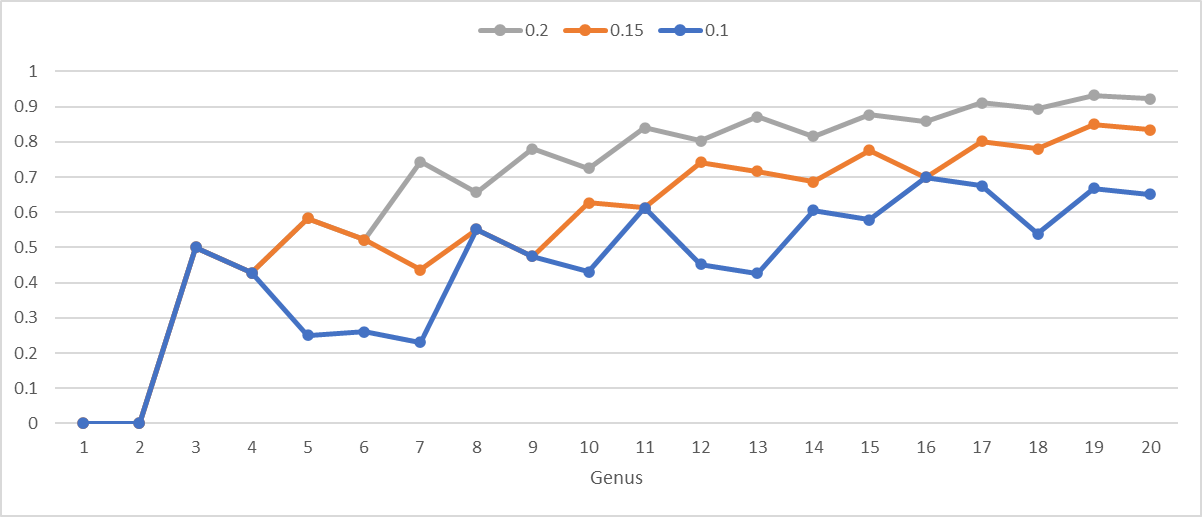}
    \caption{Proportion of $S \in \SS_g$ with $|t(S)-(1-\gamma)g|<\epsilon g$. \newline
    Plotted for $\epsilon=0.2$, $\epsilon=0.15$ and $\epsilon=0.1$.}
    \label{fig:my_label}
\end{figure}

\begin{figure}
    \centering
    \includegraphics[height=0.2\textheight]{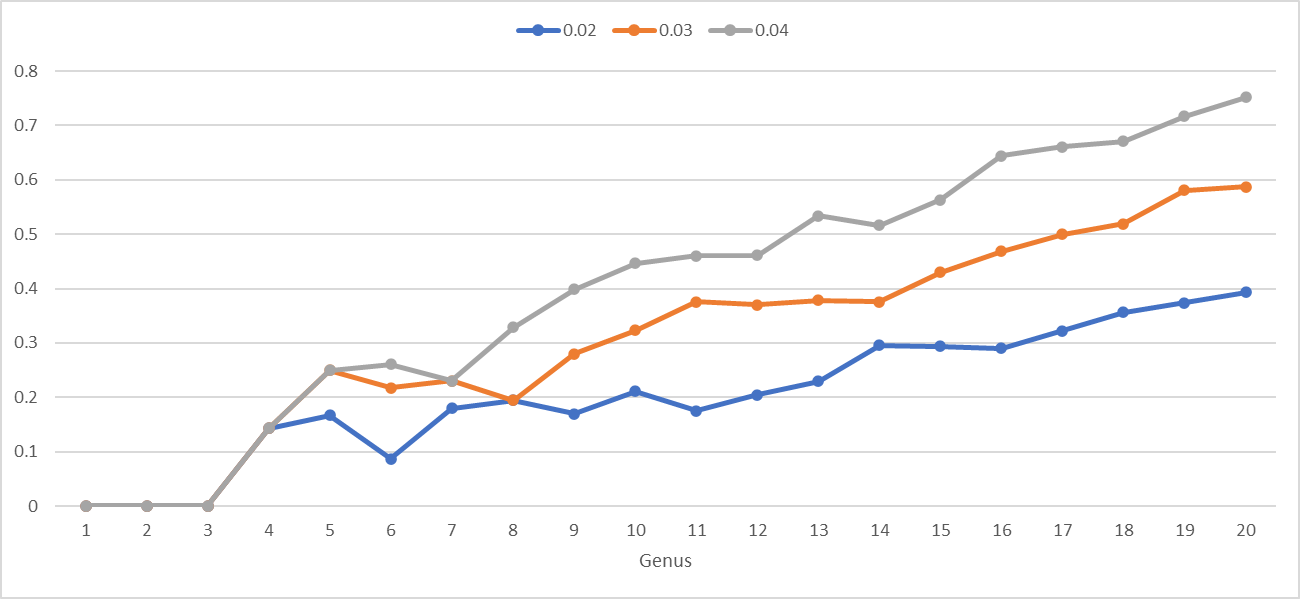}
    \caption{Proportion of $S \in \SS_g$ with $\left|w(S)-\frac{1}{10\phi} g^2 \right|<\epsilon g^2$. \newline
    Plotted for $\epsilon=0.02$, $\epsilon=0.03$ and $\epsilon=0.04$.}
    \label{fig:my_label}
\end{figure}

In order to prove Theorem \ref{main result}(1) we partition the minimal generating set $\A(S)$ into two parts.
For every numerical semigroup $S$ we have
\[
[m(S),2m(S)-1]\cap S\subseteq \A(S).
\]
Let $e_1(S)=\#\left([m(S),2m(S)-1]\cap S\right)$ and $e_2(S)=e(S)-e_1(S)$.
We separately estimate $e_1(S)$ and $e_2(S)$ for a typical numerical semigroup in $\SS_g$. Combining these estimates proves Theorem \ref{main result}(1).

\begin{proposition}\label{estimate e_1 e_2}
For fixed $\epsilon > 0$, we have
\begin{enumerate}
\item 
\[
\lim_{g\to\infty}\P_{g}\left[\left|e_1(S)-\frac{1}{\sqrt{5}} g \right|< \epsilon g\right]=1, \text{ and }
\]
\item\[
\lim_{g\to\infty}\P_{g}\left[e_2(S)< \epsilon g\right]=1.
\]
\end{enumerate}
\end{proposition}
Similarly, in order to prove Theorem \ref{main result}(2) we partition $PF(S)$ into two parts.
For any numerical semigroup $S$ with Frobenius number $F$ and multiplicity $m$, we have
\[
\Hs(S)\cap [F-m+1,F]\subseteq PF(S).
\]
Let $t_1(S)=\#(\Hs(S)\cap [F-m+1,F])$ and $t_2(S)=t(S)-t_1(S)$.
We separately estimate $t_1(S)$ and $t_2(S)$ for a typical numerical semigroup in $\SS_g$. Combining these estimates proves Theorem \ref{main result}(2).

\begin{proposition}\label{estimate t_1 and t_2}
For fixed $\epsilon > 0$, we have
\begin{enumerate}
    \item \[
\lim_{g\to\infty}\P_{g}\left[\left|t_1(S)-(1-\gamma) g \right|< \epsilon g\right]=1, \text{ and }
\]
\item \[
\lim_{g\to\infty}\P_{g}\left[t_2(S)< \epsilon g\right]=1.
\]
\end{enumerate}
\end{proposition}

\begin{figure}
    \centering
    \includegraphics[height=0.2\textheight]{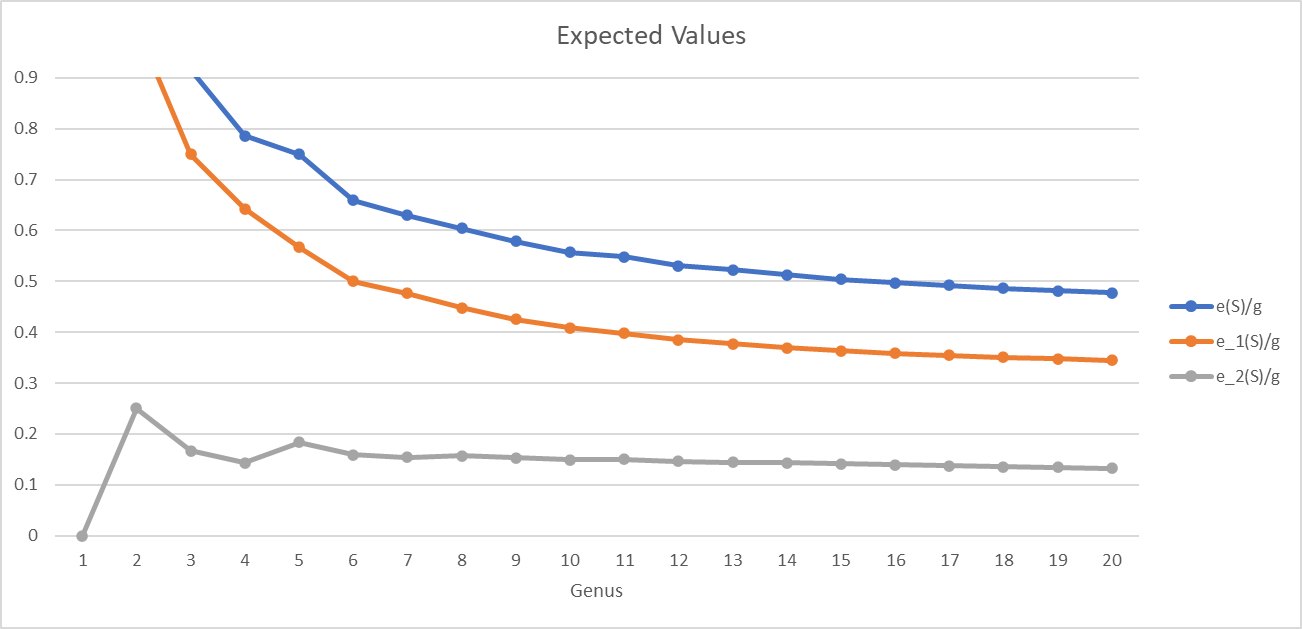}
    \caption{Expected values of $\frac{e(S)}{g}$, $\frac{e_1(S)}{g}$ and $\frac{e_2(S)}{g}$ taken over $S\in\SS_g$.}
    \label{fig:my_label}
\end{figure}

\begin{figure}
    \centering
    \includegraphics[height=0.2\textheight]{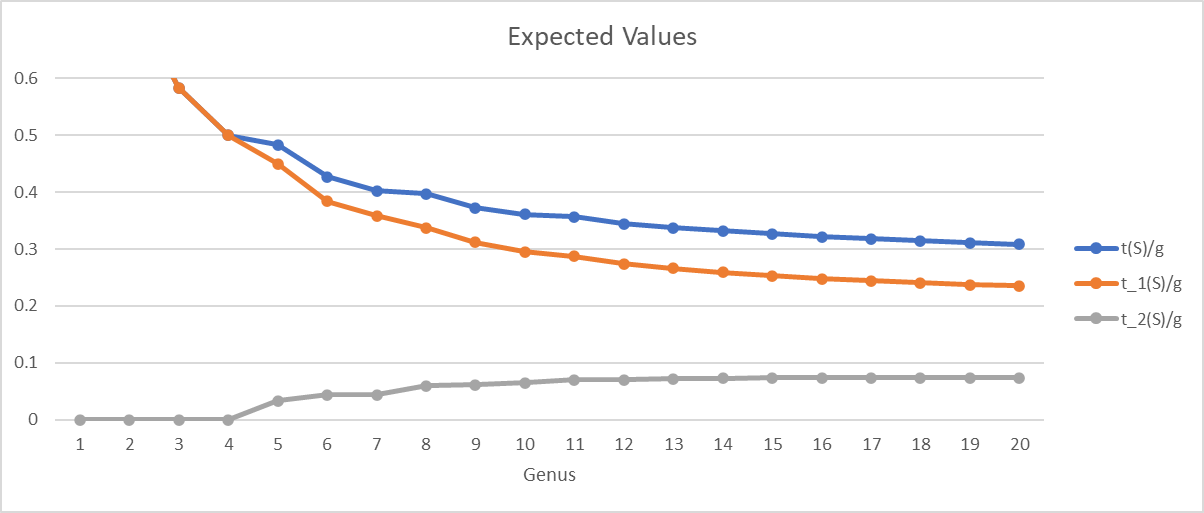}
    \caption{Expected values of $\frac{t(S)}{g}$, $\frac{t_1(S)}{g}$ and $\frac{t_2(S)}{g}$ taken over $S\in\SS_g$.}
    \label{fig:my_label}
\end{figure}

\subsection{Counting Numerical Semigroups with Large Invariants}

Among numerical semigroups in $\SS_g$, we have seen that most have 
\begin{itemize}
\item multiplicity close to $\gamma g$,
\item Frobenius number close to $2\gamma g$, 
\item embedding dimension close to $\frac{1}{\sqrt{5}}g$,  
\item type close to $(1-\gamma)g$, and
\item weight close to $\frac{1}{10\varphi} g^2$.
\end{itemize}
One could also ask about extreme values of these invariants, and to try to count numerical semigroups in $\SS_g$ with invariants close to these maximum or minimum values.  Basic properties of numerical semigroups imply that for $S\in\SS_g,\ m(S)\leq g+1,\  F(S)\leq 2g-1,\ e(S)\leq g+1$, and $t(S) \leq g$.

Numerical semigroups for which $F(S)=2g(S)-1$ are called \emph{symmetric}. Backelin has studied the problem of counting symmetric numerical semigroups \cite{Backelin}.
\begin{thm}\cite[Proposition 1]{Backelin}
For $i\in\{0,1,2\}$, the following limit exists and is positive:
\[
\lim_{\substack{g\to\infty\\ g\equiv i\pmod*{3}}}\frac{\#\{S\in\SS_g\mid F(S)=2g-1\}}{\sqrt[3]{2}^g}.
\]
\end{thm}

Kaplan has studied the problem of counting $S\in\SS_g$ with $m(S)=g-k$ for fixed $k$ \cite{Kaplan}.
\begin{thm}\cite[Proposition 13]{Kaplan}\label{thm:count_mlarge}
For each $k\geq 0$, there is a monic polynomial $f_k(x) \in \Q[x]$ of degree $k+1$ such that for $g>3k$ we have
\[
\#\{S\in\SS_g\mid m(S)=g-k\}=\frac{1}{(k+1)!}f_k(g).
\]
\end{thm}
The following fact is not stated in \cite{Kaplan}, so we provide a proof here.
\begin{corollary}
The polynomials $f_k(x)$ have integer coefficients.
\end{corollary}
\begin{proof}
Define polynomials $F_k(x)=\frac{1}{(k+1)!}f_k(x+3k+1)$.  Therefore $F_k(x)\in\Q[x]$ has degree $k+1$ and $F_k(n)\in \Z$ for all $n\in\N_0$. Fix $k$ and recursively define $a_i$ for $0\leq i\leq k+1$ as follows. Let $a_0=F_k(0)$ and $a_i=F_k(i)-\sum_{j=0}^{i-1}a_j\binom{i}{j}$. It is clear that each $a_i\in\Z$. Now $F_k(x)$ and $\sum_{i=0}^{k+1} a_i\binom{x}{i}$ are two polynomials of degree $k+1$, whose values match at $k+2$ points. It follows that
$$F_k(x)=\sum_{i=0}^{k+1} a_i\binom{x}{i},$$
and hence $(k+1)!F_k(x)\in \Z[x]$. Therefore, $f_k(x)\in \Z[x]$.
\end{proof}

Singhal has studied the problem of counting $S\in\SS_g$ with $t(S) = g-k$ for fixed $k$ \cite{Large type}.
\begin{thm}\cite[Theorem 1.7]{Large type}
For each $k\geq 0$ there is a positive integer $c_k$ such that for $g\geq 3k-1$, 
\[
\#\{S\in\SS_g\mid t(S)=g-k\}=c_k.
\]
\end{thm}
We prove an analogous result for numerical semigroups in $\SS_g$ with embedding dimension close to $g$.
\begin{thm}\label{large embedding dimension}
For each $k\geq -1$, there is a polynomial $h_k(t)\in\Q[t]$ of degree $\lceil\frac{k}{2}\rceil$ such that for all $g\geq \frac{9k+7}{2}$ we have
\[
\#\{S\in\SS_g\mid e(S)=g-k\}=h_k(g).
\]
Moreover $\lceil\frac{k}{2}\rceil!h_k(t)$ is a monic polynomial with integer coefficients.
\end{thm}

\subsection{Outline of the Paper}
In Section \ref{sec:F3m} we review results of Zhao that characterize numerical semigroups in $\SS_g$ with $F(S)< 3m(S)$.  In Section \ref{sec:randomvariables} we prove several results about random variables on $\SS_g$ and show how to deduce Corollary \ref{main_cor} from Theorem \ref{main result}.  In Section \ref{sec:embedding} we prove Proposition \ref{estimate e_1 e_2} and in Section \ref{sec:type} we prove Proposition \ref{estimate t_1 and t_2}.  In Section \ref{sec:ind_prob} we prove a result about the probability that a subset of elements is contained in a random element of $\SS_g$.  We use this result in Section \ref{Sec: weight} to prove Theorem \ref{Weight result}. In Section \ref{Sec: large ED} we prove Theorem~\ref{large embedding dimension}.

\section{Numerical semigroups with $F(S)<3m(S)$}\label{sec:F3m} 
A major step in Zhai's proof of Theorem \ref{thm_zhai} is to prove a conjecture of Zhao \cite[Conjecture 4.1]{Zhao}, which states that 
\[
\lim_{g\to\infty}\P_g[F(S)<3m(S)]=1.
\]
We define the following two subsets of $\SS_g$:
\begin{eqnarray*}
\B(g) & = & \{S\in\SS_g \mid F(S)<2m(S)\},\\
\cC(g)& = & \{S\in\SS_g\mid 2m(S)<F(S)<3m(S)\}.
\end{eqnarray*}
We further divide up the elements of $\B(g)$ by multiplicity.  Let
\[
\B(g,m)=\{S\in\SS_g\mid m(S)=m,F(S)<2m\}.
\]
Throughout this paper when describing a numerical semigroup by listing its elements, we use the symbol $\rightarrow$ to indicate that it contains all larger elements.  For example, the numerical semigroup of genus $g$ containing all positive integers larger than $g$ is $S = \{0,g+1 \rightarrow\}$.
\begin{proposition}\cite[Corollary 2.2]{Zhao}\label{Zhao F<2m}
Numerical semigroups in $\B(g,m)$ are in bijection with subsets $B\subseteq \{1,2,\dots,m-1\}$ of size $2m-g-2$. The bijection is as follows.  Given such a subset $B$, let
\[
S_{m,B}=(m+B)\cup\{0,m,2m\rightarrow\} \in \B(g,m).
\]
\end{proposition}
Note that $\B(g,m)\neq\emptyset$ if and only if $\frac{g}{2}+1\leq m\leq g+1$.
We further divide up the elements of $\cC(g)$, first by $F(S)-2m(S)$ and then by multiplicity.  For a fixed positive integer $k$, we define the following two sets: 
\begin{eqnarray*}
\cC(k,g) & = & \{S\in\cC(g)\mid F(S)=2m(S)+k\},\\
\cC(m,k,g) & = & \{S\in\cC(k,g)\mid m(S)=m\}.
\end{eqnarray*}
Zhao counts numerical semigroups with $2m(S) < F(S) < 3m(S)$ by dividing them up by \emph{type} \cite[Section 3.1]{Zhao}. (Note that this use of type is unrelated to how we have used it earlier in the paper.) Let
\[
A_k=\{A\subseteq [0,k-1]\mid 0\in A, k\notin A+A\}.
\]
For $A \in A_k$, we define the following two sets:
\begin{eqnarray*}
\cC(k,A,g) & = & \{S\in \cC(k,g)\mid m(S)+A=S\cap [m(S),m(S)+k]\},\\
\cC(m,k,A,g)& = &\{S\in\cC(k,A,g)\mid m(S)=m\}.
\end{eqnarray*}
\begin{proposition}\cite[Proposition 3.3, Corollary 3.4]{Zhao}\label{Zhao F<3m}
For $g\geq 3k$, numerical semigroups in $\cC(m,k,A,g)$ are in bijection with subsets $B\subseteq [m+k+1,2m+k-1]\setminus (2m+A+A)$ of size $2m-g+k-|A|-|(A+A)\cap [0,k]|$.
The bijection is as follows.  Given such a subset $B$, let
\[
S_{m,A,B}=\{0\}\cup (m+A) \cup \big(2m+((A+A)\cap [0,k])\big) \cup B \cup \{2m+k+1\rightarrow\}\in\cC(m,k,A,g).
\]
\end{proposition}

\section{Random Variables on $\SS_g$}\label{sec:randomvariables}

In this section we prove several results about nonnegative random variables on $\SS_g$ and show how to deduce Corollary \ref{main_cor} from Theorem \ref{main result}.

\begin{lemma}\label{Lemma Average to 0}
Suppose we have a sequence of nonnegative random variables $X_g$ on $\SS_g$ such that
\[
\lim_{g\to\infty}\frac{1}{g^n}\E_g[X_g]=0.
\]
Then for every $\epsilon>0$
\[
\lim_{g\rightarrow \infty} \P_g[X_g(S)<\epsilon g^n]=1.
\]
\end{lemma}
\begin{proof}
Assume for the sake of contradiction that
$$\liminf_{g\rightarrow \infty} \P_g[X_g(S)<\epsilon g^n]<1.$$
Pick $0<\delta<1$ such that
$$\liminf_{g\rightarrow \infty}\P_g[X_g(S)<\epsilon g^n]  <1-\delta.$$
This implies that we have a sequence $g_i$, such that $\lim_{i\to\infty}g_i=\infty$ and for all $i$ we have
\[
\P_{g_i}[X_{g_i}(S)\geq\epsilon g_i^n]\geq  \delta.
\]
Therefore, we see that for all $i$, we have
\[
\E_{g_i}[X_{g_i}]\geq \delta \epsilon g_i^n.
\]
This contradicts the fact that
\[
\lim_{g\to\infty}\frac{1}{g^n}\E_g[X_g]=0.\qedhere
\]
\end{proof}

\begin{lemma}\label{Lemma Average}
Let $X_g$ be a sequence of nonnegative random variables on $\SS_g$. Suppose that there is a positive integer $n$ and constant $M$ such that for every $g$ and every $S\in\SS_g$, we have $X_g(S)\leq Mg^{n}$. 
Suppose further that there is a $\beta$ such that for every $\epsilon>0$, we have
\[
\lim_{g\to\infty}\P_g[|X_g(S)-\beta g^n|<\epsilon g^n]=1.
\]
Then,
\[
\lim_{g\to\infty}\frac{1}{g^n}\E_{g}[X_g]=\beta.
\]
\end{lemma}
\begin{proof}
Fix $\epsilon_1,\epsilon_2>0$. We know that
\[
\lim_{g\to\infty}\P_g[|X_g(S)-\beta g^n|<\epsilon_1 g^n]=1.
\]
This means there is $M_1>0$ such that for $g>M_1$ we have
$$\P_g[|X_g(S)-\beta g^n|<\epsilon_1 g^n]>1-\epsilon_2.$$
This implies that for $g>M_1$, we have
$$\frac{1}{g^n}\E_{g}[X_g] =\frac{1}{g^n}\frac{1}{N(g)}\sum_{S\in\SS_g}X_g(S)
    \geq \frac{1}{g^n}\frac{1}{N(g)}(1-\epsilon_2)N(g)\left(\beta-\epsilon_1\right)g^n
    =(1-\epsilon_2)\left(\beta-\epsilon_1\right).$$
For $g>M_1$, we also have
\begin{equation*}
\begin{split}
    \frac{1}{g^n}\E_{g}[X_g]
    &= \frac{1}{g^n}\frac{1}{N(g)}\sum_{S\in\SS_g}X_g(S)\\
    &\leq \frac{1}{g^n}\frac{1}{N(g)}N(g)\left(\beta+\epsilon_1\right)g^n
    +\frac{1}{g^n}\frac{1}{N(g)}\epsilon_2N(g)Mg^n\\
    &=\left(\beta+\epsilon_1\right) +\epsilon_2 M.
\end{split}
\end{equation*}
Since $\epsilon_1$ and $\epsilon_2$ were arbitrary we see that
\[\lim_{g\to\infty}\frac{1}{g^n}\E_{g}[X_g]=\beta.\qedhere\]
\end{proof}
We now show how to apply this result to determine the expected value of certain invariants of numerical semigroups.
\begin{proof}[Proof that Theorem \ref{main result} implies Corollary \ref{main_cor}]
Since $e(S)\leq g(S)+1$, we see that Theorem \ref{main result}(1) and Lemma \ref{Lemma Average} imply Corollary \ref{main_cor}(1).  Similarly, since $t(S)\leq g(S)$, we see that Theorem \ref{main result}(2) and Lemma \ref{Lemma Average} imply Corollary \ref{main_cor}(2). 
\end{proof}

We apply the following result in Section \ref{Sec: weight} about the distribution of weights of $S\in \SS_g$.
\begin{lemma}\label{Concentration by Chebychev}
Suppose we have a sequence of random variables $X_g$ on $\SS_g$. Suppose further that there is a positive integer $n$ and constant $\beta$ such that 
\[
\lim_{g\to\infty}\frac{1}{g^n}\E_{g}[X_g]=\beta,\;\;\; \lim_{g\to\infty}\frac{1}{g^{2n}}\E_{g}[X_g^2]=\beta^2.
\]
Then for every $\epsilon>0$, we have
\[
\lim_{g\to\infty}\P_g[|X_g(S)-\beta g^n|<\epsilon g^n]=1.
\]
\end{lemma}
\begin{proof}
Fix $\epsilon>0$.
We have $\E_{g}[X_g]=\beta g^n+o(g^n)$ and $\E_{g}[X_g^2]=\beta^2 g^{2n}+o(g^{2n})$.
Therefore, $\Var_{g}[X_g]=o(g^{2n})$.
This means given $\epsilon_1>0$, there is an $M>0$ such that for all $g>M$, we have
\begin{eqnarray*}
\Var_{g}[X_g]& < &  \epsilon_1 g^{2n},\ \text{and}\\
|\E_{g}[X_g]- \beta g^n|& < &\frac{\epsilon}{2}g^n.
\end{eqnarray*}
By Chebychev's inequality, we see that for $g>M$,
\[
\P_g[|X_g(S)-\beta g^n| >\epsilon g^n]
\leq\P_g\Big[|X_g(S)-\E[X_g(S)]|>\frac{\epsilon}{2} g^n\Big]
\leq \frac{4\Var_{g}[X_g]}{\epsilon^2 g^{2n}}\leq \frac{4}{\epsilon^2}\epsilon_1.
\]
We conclude that
\[
\lim_{g\to\infty}\P_g[|X_g(S)-\beta g^n|<\epsilon g^n]=1.\qedhere
\]
\end{proof}

\section{Embedding dimension of a typical numerical semigroup}\label{sec:embedding} 

The goal of this section is to prove Proposition \ref{estimate e_1 e_2}.  We first prove the part about the typical size of $e_1(S)$.

\begin{proof}[Proof of Proposition \ref{estimate e_1 e_2}(1)]
For $S\in\SS_g$ we have $\#\left(\Hs(S)\cap [1,m(S)-1]\right)=m(S)-1$ and 
\[
\#\left(\Hs(S)\cap [2m(S)+1,F(S)]\right)\leq\max(0, F(S)-2m(S))\leq |F(S)-2m(S)|.
\]
It follows that
\[
0\leq g-(m(S)-1)-\#(\Hs(S)\cap [m(S),2m(S)-1])\leq |F(S)-2m(S)|.
\]
Note that $\#(\Hs(S)\cap [m(S),2m(S)-1]) = m(S) - e_1(S)$.  This implies that
\[
2m(S)-g-1\leq e_1(S) \leq 2m(S)-g-1+|F(S)-2m(S)|.
\]
Therefore,
\[
\Big|2m(S)-g-e_1(S)\Big|\leq 1+|F(S)-2m(S)|.
\]
We note that $2\gamma-1 =\frac{1}{\sqrt{5}}$ and conclude that
\begin{equation*}
\begin{split}
    \left|e_1(S)-\frac{1}{\sqrt{5}} g\right|
    &\leq \Big|e_1(S)-(2m(S)-g)\Big| +\Big|(2m(S)-g)-(2\gamma-1)g\Big|\\
    &\leq 1+|F(S)-2m(S)| +2|m(S)-\gamma g|.
\end{split}  
\end{equation*}
We see that for $g>\frac{3}{\epsilon}$, we have
\[
\P_g\left[\left|e_1(S)-\frac{1}{\sqrt{5}} g\right| \geq \epsilon g\right] 
\leq \P_g\left[\left|F(S)-2m(S)\right| \geq \frac{\epsilon}{3} g\right]
+\P_g\left[\left|m(S)-\gamma g\right| \geq \frac{\epsilon}{6} g\right].
\]
The result now follows from Theorem \ref{Kaplan Mult}.
\end{proof}

Next we bound $e_2(S)$ for numerical semigroups with $F(S)<2m(S)$. Let $F_n$ denote the $n$\textsuperscript{th} Fibonacci number, where $F_1 = F_2 = 1$, and $F_{n+2} = F_{n+1} + F_{n}$ for all $n \ge 1$. Recall that
\[
F_n = \frac{1}{\sqrt{5}}\left(\varphi^n - (1-\varphi)^n\right).
\]
\begin{proposition}
For any $g \ge 0$ we have
\[
\sum_{S\in\B(g)}e_2(S)\leq 2F_{g+1}.
\]
\end{proposition}
\begin{proof}
In this proof we see the first instance of a style of argument that will appear several times later in this paper, so we give an outline of the strategy.  If $F(x)$ is a polynomial we write $[x^m]\left(F(x)\right)$ for its $x^m$-coefficient.  We show that the following estimate holds.\\
\noindent{\textbf{Claim}}: 
\begin{equation}\label{eqn:e2_sum}
\sum_{S \in \B(g,m)} e_2(S) \le 2\cdot [x^{g-m}]\left((1+x)^{m}-x^{\lfloor\frac{m}{2}\rfloor}(x+2)^{\lfloor\frac{m}{2}\rfloor}(1+x)^{m-2\lfloor\frac{m}{2}\rfloor}\right).
\end{equation}
Assuming this for now, we complete the proof of the proposition.  Since 
\[
x^{\lfloor\frac{m}{2}\rfloor}(x+2)^{\lfloor\frac{m}{2}\rfloor}(1+x)^{m-2\lfloor\frac{m}{2}\rfloor}
\]
has nonnegative coefficients, we see that 
\[
\sum_{S \in \B(g,m)} e_2(S) \le 2\cdot [x^{g-m}]\left((1+x)^{m}\right) = 2 \binom{m}{g-m}.
\]
For $S \in \B(g)$, it is clear that $m(S) \in [\lceil\frac{g}{2}\rceil+1,g+1]$.  Taking a sum over $m(S)$ gives
\[
\sum_{S\in\B(g)}e_2(S)
=
\sum_{m=\lceil\frac{g}{2}\rceil+1}^{g+1} \sum_{S \in \B(g,m)} e_2(S) 
\leq 2\sum_{m=\lceil\frac{g}{2}\rceil+1}^{g+1}\binom{m}{g-m} \le  2F_{g+1}.
\]

We now prove the inequality \eqref{eqn:e2_sum}.  We have
\[
\A(S_{m,B})=\{m\}\cup \{m+i\mid i\in B\}\cup \{2m+j\mid 1\leq j\leq m-1,j\notin B,j\notin B+B\},
\]
and so
\[
e_2(S_{m,B})=\#\{2m+j\mid 1\leq j\leq m-1,j\notin B,j\notin B+B\}.
\]
For $j\in [1,m-1]$, $2m+j\in\A(S_{m,B})$ if and only if $j\not\in B\cup(B+B)$.
Let $j_1=\lfloor\frac{j-1}{2}\rfloor$.
A necessary condition for $2m+j\in\A(S_{m,B})$ is that $j\notin B$ and none of $\{1,j-1\},\{2,j-2\},\dots,\{j_1,j-j_1\}$ is a subset of $B$. By inclusion-exclusion we see that
\begin{equation*}
    \begin{split}
    \#\{S\in\B(g,m)\mid 2m+j\in \A(S)\}
    &\leq\sum_{l=0}^{j_1}(-1)^{l}\binom{j_1}{l}\binom{m-2-2l}{2(m-1)-g-2l}\\
    &=\sum_{l=0}^{j_1}(-1)^{l}\binom{j_1}{l}\binom{m-2-2l}{g-m}.
    \end{split}
\end{equation*}
Taking a sum over these terms gives
\[
\sum_{S\in\B(g,m)}e_2(S) \leq 2\sum_{j_1=0}^{\lfloor\frac{m}{2}\rfloor-1}\sum_{l=0}^{j_1}(-1)^{l}\binom{j_1}{l}\binom{m-2-2l}{g-m}.
\]
Now, notice that 
\[
\sum_{l=0}^{j_1}(-1)^{l}\binom{j_1}{l}\binom{m-2-2l}{g-m} \\
=  [x^{g-m}]\left( 
(1+x)^{m-2}\sum_{l=0}^{j_1}(-1)^l\binom{j_1}{l}(1+x)^{-2l}\right).
\]
Since 
\[
(1+x)^{m-2}\sum_{l=0}^{j_1}(-1)^l\binom{j_1}{l}(1+x)^{-2l}
=(1+x)^{m-2}\left(1-\frac{1}{(1+x)^2}\right)^{j_1}=(1+x)^{m-2}\frac{x^{j_1}(x+2)^{j_1}}{(x+1)^{2j_1}},
\]
we see that
\[
\sum_{j_1=0}^{\lfloor\frac{m}{2}\rfloor-1}\sum_{l=0}^{j_1}(-1)^{l}\binom{j_1}{l}\binom{m-2-2l}{g-m} 
=
[x^{g-m}]\left(\sum_{j_1=0}^{\lfloor\frac{m}{2}\rfloor-1}(1+x)^{m-2}\frac{x^{j_1}(x+2)^{j_1}}{(x+1)^{2j_1}}
\right).
\]
Noting that 
\begin{eqnarray*}
    \sum_{j_1=0}^{\lfloor\frac{m}{2}\rfloor-1}(1+x)^{m-2}\frac{x^{j_1}(x+2)^{j_1}}{(x+1)^{2j_1}}
    &= &\frac{(1+x)^{m}}{(1+x)^{2}}\frac{\left(1-\left(\frac{x(x+2)}{(1+x)^2}\right)^{\lfloor\frac{m}{2}\rfloor}\right)}{\left(1-\frac{x(x+2)}{(1+x)^2}\right)}\\
    &= &(1+x)^{m}-x^{\lfloor\frac{m}{2}\rfloor}(x+2)^{\lfloor\frac{m}{2}\rfloor}(1+x)^{m-2\lfloor\frac{m}{2}\rfloor}
\end{eqnarray*}
completes the proof.
\end{proof}

We now give a similar, but more complicated, argument to bound $e_2(S)$ for numerical semigroups with $2m(S)<F(S)<3m(S)$.

\begin{lemma}\label{e2_Ckg}
For any positive integers $g$ and $k$, we have
\[
\sum_{S\in\cC(k,g)}e_2(S)\leq 2F_{g+k}.
\]
\end{lemma}
\begin{proof}
A major step in the proof is to prove the following inequality.\\
\noindent{\textbf{Claim}}:
\begin{equation}\label{eqn:e2_sum_Ckg}
\sum_{S\in\cC(k,g)}e_2(S)\le 2\cdot [x^{g-m}]\left(
(1+x)^{m+k+1}-x^{\lfloor\frac{m+k+1}{2}\rfloor}(x+2)^{\lfloor\frac{m+k+1}{2}\rfloor}(1+x)^{m+k+1-2\lfloor\frac{m+k+1}{2}\rfloor}
\right).
\end{equation}
Assuming this for now, we complete the proof of the lemma. Since
\[
x^{\lfloor\frac{m+k+1}{2}\rfloor}(x+2)^{\lfloor\frac{m+k+1}{2}\rfloor}(1+x)^{m+k+1-2\lfloor\frac{m+k+1}{2}\rfloor}
\]
has nonnegative coefficients, we see that 
\[
\sum_{S\in\cC(k,g)}e_2(S)\le 2\cdot [x^{g-m}]\left(
(1+x)^{m+k+1}\right) 
= \binom{m+k-1}{g-m}.
\]
Taking a sum over $m$ shows that
\[
\sum_{S\in\cC(k,g)}e_2(S)\leq 2\sum_{m}\binom{m+k-1}{g-m}= 2F_{g+k},
\]
completing the proof.

We now prove the inequality \eqref{eqn:e2_sum_Ckg}. Numerical semigroups $S$ in $\cC(m,k,g)$ are determined by a subset $B=S\cap [m+1,2m+k-1]$ of size $2m-g+k-1$.
For $j\in [1,m+k]$, let $j_1=\lfloor\frac{j-1}{2}\rfloor$.
A necessary condition for $2m+j\in\A(S)$ is that none of $\{m+1,m+j-1\},\dots,\{m+j_1,m+j-j_1\}$ is a subset of $B$.
This means that we can bound the number of $S\in\cC(m,k,g)$ with $2m+j\in\A(S)$ by the number of subsets $B\subseteq[m+1,2m+k-1]$ of size $2m-g+k-1$ for which none of $\{m+1,m+j-1\},\dots,\{m+j_1,m+j-j_1\}$ is a subset of $B$. 
By inclusion-exclusion we see that
\begin{equation*}
    \begin{split}
    \#\{S\in\cC(m,k,g)\mid 2m+j\in \A(S)\}
    &\leq\sum_{l=0}^{j_1}(-1)^{l}\binom{j_1}{l}\binom{m+k-1-2l}{2m-g+k-1-2l}\\
    &=\sum_{l=0}^{j_1}(-1)^{l}\binom{j_1}{l}\binom{m+k-1-2l}{g-m}.
    \end{split}
\end{equation*}
Taking a sum over these terms gives
\[
\sum_{S\in\cC(m,k,g)}e_2(S) \leq 2\sum_{j_1=0}^{\lfloor\frac{m+k-1}{2}\rfloor}\sum_{l=0}^{j_1}(-1)^{l}\binom{j_1}{l}\binom{m+k-1-2l}{g-m}.
\]
Now, notice that
\[
\sum_{l=0}^{j_1}(-1)^{l}\binom{j_1}{l}\binom{m+k-1-2l}{g-m} 
=
[x^{g-m}]\left( 
(1+x)^{m+k-1}\sum_{l=0}^{j_1}(-1)^l\binom{j_1}{l}(1+x)^{-2l}
\right).
\]
Since,
\begin{equation*}
    \begin{split}
    (1+x)^{m+k-1}\sum_{l=0}^{j_1}(-1)^l\binom{j_1}{l}(1+x)^{-2l}
    &=(1+x)^{m+k-1}\left(1-\frac{1}{(1+x)^2}\right)^{j_1}\\
    &=(1+x)^{m+k-1}\frac{x^{j_1}(x+2)^{j_1}}{(x+1)^{2j_1}},    
    \end{split}
\end{equation*}
we see that
\[
\sum_{j_1=0}^{\lfloor\frac{m+k-1}{2}\rfloor}\sum_{l=0}^{j_1}(-1)^{l}\binom{j_1}{l}\binom{m+k-1-2l}{g-m} 
=
[x^{g-m}]\left(
 \sum_{j_1=0}^{\lfloor\frac{m+k-1}{2}\rfloor}(1+x)^{m+k-1}\frac{x^{j_1}(x+2)^{j_1}}{(x+1)^{2j_1}}
\right).
\]
Noting that 
\begin{eqnarray*}
   & &  \sum_{j_1=0}^{\lfloor\frac{m+k-1}{2}\rfloor}(1+x)^{m+k-1}\frac{x^{j_1}(x+2)^{j_1}}{(x+1)^{2j_1}}
   =\frac{(1+x)^{m+k+1}}{(1+x)^{2}}\frac{\left(1-\left(\frac{x(x+2)}{(1+x)^2}\right)^{\lfloor\frac{m+k-1}{2}\rfloor+1}\right)}{\left(1-\frac{x(x+2)}{(1+x)^2}\right)} \\
    & = &(1+x)^{m+k+1}-x^{\lfloor\frac{m+k+1}{2}\rfloor}(x+2)^{\lfloor\frac{m+k+1}{2}\rfloor}(1+x)^{m+k+1-2\lfloor\frac{m+k+1}{2}\rfloor}
\end{eqnarray*}
completes the proof of \eqref{eqn:e2_sum_Ckg}.
\end{proof}

\begin{lemma}\label{Average of e2}
We have
\[
\lim_{g\to\infty}\frac{1}{g}\E_g[e_2]=0.
\]
\end{lemma}

Proposition \ref{estimate e_1 e_2}(2) follows directly from this result together with Lemma \ref{Lemma Average to 0}.  Therefore, proving this result completes the proof of Theorem \ref{main result}(1).
\begin{proof}
Choose $\epsilon>0$ and consider the $M(\epsilon)$ given by Proposition \ref{F 2m}. For any $g$, applying Lemma \ref{e2_Ckg} gives
\begin{equation*}
    \begin{split}
    \sum_{S\in\SS_g}e_2(S)&= \sum_{S\in\B(g)}e_2(S)+\sum_{k=1}^{M(\epsilon)}\sum_{S\in\cC(k,g)}e_2(S)+\sum_{\substack{S\in\SS_g\\ F(S)> 2m(S)+M(\epsilon)}}e_2(S)\\
    &\leq 2F_{g+1}+\sum_{k=1}^{M(\epsilon)}2F_{g+k}+\sum_{\substack{S\in\SS_g\\ F(S)> 2m(S)+M(\epsilon)}}(g+1).
    \end{split}
\end{equation*}
Noting that $F_n < \frac{\varphi^n+1}{\sqrt{5}}$ and applying Proposition \ref{F 2m} to the last term in this expression gives
\begin{eqnarray*}
\sum_{S\in\SS_g}e_2(S)  &< & 2\frac{\phi^{g+1}+1}{\sqrt{5}}+2\sum_{k=1}^{M(\epsilon)}\frac{\phi^{g+k}+1}{\sqrt{5}}+(g+1)\epsilon N(g)\\
    &= &\phi^{g+1} \frac{2}{\sqrt{5}}\left(1+\frac{\phi^{M(\epsilon)}-1}{\phi-1}\right)+\frac{2}{\sqrt{5}}(M(\epsilon)+1)+(g+1)\epsilon N(g).
\end{eqnarray*}
Therefore,
\[
\limsup_{g\to\infty}\frac{1}{g}\frac{1}{N(g)}\sum_{S\in\SS_g}e_2(S)\leq \epsilon.
\]
Since $\epsilon$ was arbitrary, we see that
\[
\lim_{g\to\infty}\frac{1}{g}\frac{1}{N(g)}\sum_{S\in\SS_g}e_2(S)=0.\qedhere
\]
\end{proof}

\section{Type of a typical numerical semigroup}\label{sec:type}
The goal of this section is to prove Proposition \ref{estimate t_1 and t_2}.  We first prove the part about the typical size of $t_1(S)$.  The arguments in this section are quite similar to the arguments in Section \ref{sec:embedding}.
\begin{proof}[Proof of Proposition \ref{estimate t_1 and t_2}(1)]
Let $S \in \SS_g$. We know that 
\begin{eqnarray*}
\#\left(\Hs(S)\cap [1,m(S)-1]\right) & = & m(S)-1, \\
\#\left(\Hs(S)\cap [F(S)-m(S)+1,F(S)]\right) & = & t_1(S),
\end{eqnarray*}
and
\[
\#\left(\Hs(S)\cap [m(S),F(S)-m(S)]|\right)  \leq  \max(0, F(S)-2m(S))\leq |F(S)-2m(S)|.
\]
Therefore,
\[
|g-(m(S)-1)-t_1(S)| \leq  |F(S)-2m(S)|,
\]
and 
\begin{eqnarray*}
    |t_1(S)-(1-\gamma)g|
    &\leq&  |t_1(S)-g+m(S)-1|+1 + |m(S)-\gamma g|\\
    &\leq& 1+|F(S)-2m(S)|+ |m(S)-\gamma g|.
\end{eqnarray*}
Therefore, for $g> \frac{3}{\epsilon}$, we have
\[
\P_{g}\left[\left|t_1(S)-(1-\gamma) g \right|\geq \epsilon g\right]
\leq \P_{g}\left[|F(S)-2m(S)|\geq \frac{\epsilon}{3} g\right]
+\P_{g}\left[|m(S)-\gamma g|\geq \frac{\epsilon}{3} g\right].
\]
The result now follows from Theorem \ref{Kaplan Mult}.
\end{proof}

Next we bound $t_2(S)$ for numerical semigroups with $F(S)<2m(S)$.
\begin{lemma}\label{bound t2 depth 2}
For any $g \ge 0$, we have
\[
\sum_{S\in\B(g)}t_2(S)\leq F_{g+4}.
\]
\end{lemma}
\begin{proof}
Suppose $S \in \B(g)$.  Since $F(S) - m(S) \le m(S)-1$, we see that 
\[
t_2(S) \le \#\left(PF(S)\cap\left[1,m(S)-1\right]\right),
\]
so it is enough to prove that 
\[
\sum_{S\in\B(g)} \#\left(PF(S)\cap\left[1,m(S)-1\right]\right) \le F_{g+4}.
\]
We divide the set $PF(S)\cap\left[1,m(S)-1\right]$ into two pieces and bound the size of each. Note that $\lceil\frac{m}{2}\rceil-1=\lfloor\frac{m-1}{2}\rfloor$.\\
\noindent{\textbf{Claim}}:
\begin{equation}\label{eqn:t_2_big}
\begin{split}
& \sum_{S\in\B(g,m)}\#\left(PF(S)\cap\left[\left\lceil\frac{m}{2}\right\rceil,m-1\right]\right)\\
= & 
[x^{2m-g-3}]\left(x^{-1}\left((1+x)^{m}-(1+x+x^2)^{m-\lceil\frac{m}{2}\rceil}(1+x)^{2\lceil\frac{m}{2}\rceil-m}\right)\right)
\end{split}
\end{equation}
and
\begin{equation}\label{eqn:t_2_small}
\begin{split}
& \sum_{S\in\B(g,m)}\#\left(PF(S)\cap\left[1,\left\lfloor\frac{m-1}{2}\right\rfloor\right]\right)\\
\leq & 
[x^{2m-g-4}]\left(x^{-1}\left((1+x)^{m-1}-(1+x+x^2)^{\lfloor\frac{m-1}{2}\rfloor}(1+x)^{(m-1)-2\lfloor\frac{m-1}{2}\rfloor}
\right)\right).
\end{split}
\end{equation}
Assuming these results for now, we complete the proof of the lemma.  Since both \[
(1+x+x^2)^{m-\lceil\frac{m}{2}\rceil}(1+x)^{2\lceil\frac{m}{2}\rceil-m}
\]
and 
\[
(1+x+x^2)^{\lfloor\frac{m-1}{2}\rfloor}(1+x)^{(m-1)-2\lfloor\frac{m-1}{2}\rfloor}
\]
have nonnegative coefficients, we see that the following inequalities hold:
\begin{eqnarray*}
\sum_{S\in\B(g,m)}\#\left(PF(S)\cap\left[\left\lceil\frac{m}{2}\right\rceil,m-1\right]\right) & \le & [x^{2m-g-2}]\left((1+x)^{m}\right) = \binom{m}{2m-g-2} \\
\sum_{S\in\B(g,m)}\#\left(PF(S)\cap\left[1,\left\lfloor\frac{m-1}{2}\right\rfloor\right]\right)
& \leq & [x^{2m-g-3}]\left((1+x)^{m-1}\right) = \binom{m-1}{2m-g-3}.
\end{eqnarray*}
Taking a sum over $m$ shows that
\begin{equation*}
    \begin{split}
\sum_{S\in\B(g)} \#\left(PF(S)\cap\left[1,m(S)-1\right]\right) 
    &\leq \sum_{m}\binom{m}{2m-g-2}+\sum_{m}\binom{m-1}{2m-g-3}\\
    &=\sum_{m}\binom{m}{g-m+2}+\sum_{m}\binom{m-1}{g-m+2}\\
    &=F_{g+3}+F_{g+2}=F_{g+4},
    \end{split}
\end{equation*}
completing the proof.

Now we only need to prove \eqref{eqn:t_2_big} and \eqref{eqn:t_2_small}. We recall the definition of $S_{m,B}$ from Proposition \ref{Zhao F<2m} and note that
\[
PF(S_{m,B})= \{m+j\mid j\in [1,m-1]\setminus B\}\cup \{j\in B\mid \forall i\in [1,m-1-j]\cap B\colon i+j\in B\}.
\]
We see that
\[
PF(S_{m,B})\cap\left[1,m-1\right]=\{j\in B\mid \forall i\in [1,m-1-j]\cap B\colon i+j\in B\}.
\]
For $j\in [\lceil\frac{m}{2}\rceil,m-1]$, we have $[1,m-1-j]\cap (j+[1,m-1-j])=\emptyset$. For $j\in [\lceil\frac{m}{2}\rceil,m-1]$, by inclusion-exclusion we have
\begin{eqnarray*}
\#\{S\in\B(g,m)\mid j\in PF(S)\}
 & = &  \sum_{l=0}^{m-1-j}(-1)^{l}\binom{m-1-j}{l}\binom{(m-1)-1-2l}{2(m-1)-g-1-l} \\
 & =& [x^{2m-g-3}]\left(\sum_{l=0}^{m-1-j}(-1)^{l}\binom{m-1-j}{l}x^l(1+x)^{m-2-2l}\right)\\
 & = & [x^{2m-g-3}]\left((1+x)^{m-2}\left(1-\frac{x}{(1+x)^2}\right)^{m-1-j}\right)\\
 & = & 
[x^{2m-g-3}]\left( (1+x)^{m-2}\left(\frac{x^2+x+1}{(1+x)^2}\right)^{m-1-j}
\right).
\end{eqnarray*}

Therefore,
\begin{eqnarray*}
& & \sum_{S\in\B(g,m)}\#\left(PF(S)\cap\left[\left\lceil\frac{m}{2}\right\rceil,m-1\right]\right)\\
& = & \sum_{j=\lceil\frac{m}{2}\rceil}^{m-1}\sum_{l=0}^{m-1-j}(-1)^{l}\binom{m-1-j}{l}\binom{(m-1)-1-2l}{2(m-1)-g-1-l} \\
& = & [x^{2m-g-3}]\left(\sum_{j=\lceil\frac{m}{2}\rceil}^{m-1}(1+x)^{m-2}\left(\frac{x^2+x+1}{(1+x)^2}\right)^{m-1-j}\right).
\end{eqnarray*}
Noting that
\begin{eqnarray*}
    \sum_{j=\lceil\frac{m}{2}\rceil}^{m-1}(1+x)^{m-2}\left(\frac{x^2+x+1}{(1+x)^2}\right)^{m-1-j}
    &= & (1+x)^{m-2}\sum_{k=0}^{m-1-\lceil\frac{m}{2}\rceil}\left(\frac{x^2+x+1}{(1+x)^2}\right)^{k}\\
    &= &\frac{(1+x)^{m}}{(1+x)^{2}}\frac{\left(1-\left(\frac{x^2+x+1}{(1+x)^2}\right)^{m-\lceil\frac{m}{2}\rceil}\right)}{\left(1-\left(\frac{x^2+x+1}{(1+x)^2}\right)\right)}\\
    &=& \frac{(1+x)^{m}-(1+x+x^2)^{m-\lceil\frac{m}{2}\rceil}(1+x)^{2\lceil\frac{m}{2}\rceil-m}}{x}
\end{eqnarray*}
completes the proof of \eqref{eqn:t_2_big}.

Consider $j\in [1,\lfloor\frac{m-1}{2}\rfloor]$, so $2j\leq m-1$. A necessary condition for $j\in PF(S_{m,B})$ is that $j,2j\in B$ and for every $i\in [1,j-1]$, if $i\in B$ then $j+i\in B$. For $j\in [1,\lfloor\frac{m-1}{2}\rfloor]$, by inclusion-exclusion we have
\begin{eqnarray*}
\#\{S\in\B(g,m)\mid j\in PF(S)\}
& \leq & \sum_{l=0}^{j-1}(-1)^{l}\binom{j-1}{l}\binom{(m-1)-2-2l}{2(m-1)-g-2-l} \\
& = & [x^{2m-g-4}] \left( \sum_{l=0}^{j-1}(-1)^{l}\binom{j-1}{l} x^{l} (1+x)^{m-3-2l} \right)\\
& = & [x^{2m-g-4}]\left( (1+x)^{m-3}\left(1-\frac{x}{(1+x)^2}\right)^{j-1}\right)\\
& = & [x^{2m-g-4}]\left(
(1+x)^{m-3}\left(\frac{x^2+x+1}{(1+x)^2}\right)^{j-1}
\right).
\end{eqnarray*}
Therefore,
\begin{eqnarray*}
\sum_{S\in\B(g,m)}\#\left(PF(S)\cap\left[1,\left\lfloor\frac{m-1}{2}\right\rfloor\right]\right)
& \leq & \sum_{j=1}^{\lfloor\frac{m-1}{2}\rfloor}\sum_{l=0}^{j-1}(-1)^{l}\binom{j-1}{l}\binom{(m-1)-2-2l}{2(m-1)-g-2-l} \\
& = & [x^{2m-g-4}]\left(
\sum_{j=1}^{\lfloor\frac{m-1}{2}\rfloor}(1+x)^{m-3}\left(\frac{x^2+x+1}{(1+x)^2}\right)^{j-1}
\right).
\end{eqnarray*}
Noting that
\begin{eqnarray*}
    \sum_{j=1}^{\lfloor\frac{m-1}{2}\rfloor}(1+x)^{m-3}\left(\frac{x^2+x+1}{(1+x)^2}\right)^{j-1}
    &=&\frac{(1+x)^{m-1}}{(1+x)^{2}}\frac{\left(1-\left(\frac{x^2+x+1}{(1+x)^2}\right)^{\lfloor\frac{m-1}{2}\rfloor}\right)}{\left(1-\left(\frac{x^2+x+1}{(1+x)^2}\right)\right)}\\
    &= &\frac{(1+x)^{m-1}-(1+x+x^2)^{\lfloor\frac{m-1}{2}\rfloor}(1+x)^{(m-1)-2\lfloor\frac{m-1}{2}\rfloor}}{x}
\end{eqnarray*}
completes the proof of \eqref{eqn:t_2_small}.
\end{proof}

Next we bound $t_2(S)$ for numerical semigroups with $2m(S)<F(S)<3m(S)$.
\begin{lemma}
\[
\sum_{S\in\cC(k,g)}t_2(S)\leq F_{g+k+3}
\]
\end{lemma}
\begin{proof}
First recall from Section \ref{sec:F3m} that for positive integers $k$ and $m$ with $k<m$, we have
\[
\cC(m,k,g) = \{S\in \SS_g \mid m(S) = m,\ F(S) = 2m(S) + k\}.
\]
An $S \in \cC(m,k,g)$ is determined by $B=S\cap [m+1,2m+k-1]$ where $|B| = 2m-g+k-1$. Let $j\in [1,m+k-1]$. If $j \in PF(S)$ then:
\vspace{-.3cm}
\begin{itemize}
    \item $j+m\in B$,
    \item for every $i\in[1,m+k-1-j]$, if $m+i\in B$ then $m+i+j\in B$.
\end{itemize}
\vspace{-.2cm}
We bound $\#\{S\in\cC(m,k,g) \mid j\in PF(S)\}$ by counting subsets $B$ satisfying these conditions. Our argument is similar to the proof of Lemma \ref{bound t2 depth 2}.  For $S \in \cC(m,k,g)$ we have 
$t_2(S) = \#\left(PF(S)\cap\left[1,m+k-1\right]\right)$. Note that $\lceil\frac{m+k}{2}\rceil-1=\lfloor\frac{m+k-1}{2}\rfloor$.  We divide the elements of $PF(S)\cap\left[1,m+k-1\right]$ into two sets and consider each separately.  \\
\noindent{\textbf{Claim}}:
\begin{equation}\label{eqn:t2_sum_cmkg_big}
\begin{split}
&\sum_{S\in\cC(m,k,g)}\#\left(PF(S)\cap\left[\left\lceil\frac{m+k}{2}\right\rceil,m+k-1\right]\right)\\
   & \le 
    [x^{2m-g+k-2}]\left(x^{-1}\left((1+x)^{m+k}-(1+x+x^2)^{m+k-\lceil\frac{m+k}{2}\rceil}(1+x)^{2\lceil\frac{m+k}{2}\rceil-m-k}\right)\right)
    \end{split}
\end{equation}
and 
\begin{equation}\label{eqn:t2_sum_cmkg_small}
\begin{split}
&\sum_{S\in\cC(m,k,g)}\#\left(PF(S)\cap\left[1,\left\lfloor\frac{m+k-1}{2}\right\rfloor\right]\right) \\
&\le 
[x^{2m-g+k-3}]\left(x^{-1}\left(
(1+x)^{m+k-1}-(1+x+x^2)^{\lfloor\frac{m+k-1}{2}\rfloor}(1+x)^{(m+k-1)-2\lfloor\frac{m+k-1}{2}\rfloor}
\right)\right).
\end{split}
\end{equation}
Assuming these results for now, we complete the proof of the lemma.  Since 
\[
(1+x+x^2)^{m+k-\lceil\frac{m+k}{2}\rceil}(1+x)^{2\lceil\frac{m+k}{2}\rceil-m-k}
\]
and 
\[
(1+x+x^2)^{\lfloor\frac{m+k-1}{2}\rfloor}(1+x)^{(m+k-1)-2\lfloor\frac{m+k-1}{2}\rfloor}
\]
have nonnegative coefficients, we see that the following inequalities hold:
\begin{eqnarray*}
\sum_{S\in\cC(m,k,g)}\#\left(PF(S)\cap\left[\left\lceil\frac{m+k}{2}\right\rceil,m+k-1\right]\right) 
& \le & 
[x^{2m-g+k-1}]\left((1+x)^{m+k}\right) \\
&  = & \binom{m+k}{2m-g+k-1}, \\
\sum_{S\in\cC(m,k,g)}\#\left(PF(S)\cap\left[1,\left\lfloor\frac{m+k-1}{2}\right\rfloor\right]\right)
& \leq & 
[x^{2m-g+k-2}]\left(
(1+x)^{m+k-1}\right) \\
&  = & \binom{m+k-1}{2m-g+k-2}.
\end{eqnarray*}
Therefore,
\begin{eqnarray*}
\sum_{S\in\cC(m,k,g)}t_2(S) & = & 
\sum_{S\in\cC(m,k,g)}\#\left(PF(S)\cap\left[1,m+k-1\right]\right)\\
&\leq &  \binom{m+k}{2m-g+k-1}+\binom{m+k-1}{2m-g+k-2}.
\end{eqnarray*}
Taking a sum over $m$ shows that
\begin{equation*}
    \begin{split}
    \sum_{S\in\cC(k,g)}t_2(S)
    &\leq \sum_{m}\binom{m+k}{2m-g+k-1}+\sum_{m}\binom{m+k-1}{2m-g+k-2}\\
    &=\sum_{m}\binom{m+k}{g-m+1}+\sum_{m}\binom{m+k-1}{g-m+1}\\
    &=F_{g+k+2}+F_{g+k+1}=F_{g+k+3}.
    \end{split}
\end{equation*}
This completes the proof.

We now need only prove \eqref{eqn:t2_sum_cmkg_big} and \eqref{eqn:t2_sum_cmkg_small}.  For $j\in [\lceil\frac{m+k}{2}\rceil,m+k-1]$, we have \[
[1,m+k-1-j]\cap (j+[1,m+k-1-j])=\emptyset.
\]
For $j\in [\lceil\frac{m+k}{2}\rceil,m+k-1]$, by inclusion-exclusion we have that $\#\{S\in\cC(m,k,g)\mid j\in PF(S)\}$ is at most
\begin{eqnarray*}
& & \sum_{l=0}^{m+k-1-j}(-1)^{l}\binom{m+k-1-j}{l}\binom{(m+k-1)-1-2l}{(2m-g+k-1)-1-l} \\
& =& [x^{2m-g+k-2}]\left(\sum_{l=0}^{m+k-1-j}(-1)^{l}\binom{m+k-1-j}{l}x^l(1+x)^{m+k-2-2l}\right)\\
& = & [x^{2m-g+k-2}]\left((1+x)^{m+k-2}\left(1-\frac{x}{(1+x)^2}\right)^{m+k-1-j}\right) \\
& = & [x^{2m-g+k-2}]\left(
(1+x)^{m+k-2}\left(\frac{x^2+x+1}{(1+x)^2}\right)^{m+k-1-j}
\right).
\end{eqnarray*}

Therefore,
\begin{eqnarray*}
& & \sum_{j=\lceil\frac{m+k}{2}\rceil}^{m+k-1}\sum_{l=0}^{m+k-1-j}(-1)^{l}\binom{m+k-1-j}{l}\binom{(m+k-1)-1-2l}{(2m-g+k-1)-1-l} \\
& = & [x^{2m-g+k-2}]
\left(
\sum_{j=\lceil\frac{m+k}{2}\rceil}^{m+k-1}(1+x)^{m+k-2}\left(\frac{x^2+x+1}{(1+x)^2}\right)^{m+k-1-j}
\right).
\end{eqnarray*}
Noting that
\begin{eqnarray*}
    & & \sum_{j=\lceil\frac{m+k}{2}\rceil}^{m+k-1}(1+x)^{m+k-2}\left(\frac{x^2+x+1}{(1+x)^2}\right)^{m+k-1-j}\\
    & = & (1+x)^{m+k-2}\sum_{s=0}^{m+k-1-\lceil\frac{m+k}{2}\rceil}\left(\frac{x^2+x+1}{(1+x)^2}\right)^{s}
    = \frac{(1+x)^{m+k}}{(1+x)^{2}}\frac{\left(1-\left(\frac{x^2+x+1}{(1+x)^2}\right)^{m+k-\lceil\frac{m+k}{2}\rceil}\right)}{\left(1-\left(\frac{x^2+x+1}{(1+x)^2}\right)\right)}\\
    &=& \frac{(1+x)^{m+k}-(1+x+x^2)^{m+k-\lceil\frac{m+k}{2}\rceil}(1+x)^{2\lceil\frac{m+k}{2}\rceil-m-k}}{x},
\end{eqnarray*}
completes the proof of \eqref{eqn:t2_sum_cmkg_big}.

Consider $j\in [1,\lfloor\frac{m+k-1}{2}\rfloor]$, so $2j\leq m+k-1$.
A necessary condition for $j\in PF(S)$ is that $j+m,2j+m\in B$ and for every $i\in [1,j-1]$, if $m+i\in B$ then $m+j+i\in B$. For $j\in [1,\lfloor\frac{m+k-1}{2}\rfloor]$, by inclusion-exclusion we have
\begin{eqnarray*}
\#\{S\in\cC(m,k,g)\mid j\in PF(S)\}
& \leq & \sum_{l=0}^{j-1}(-1)^{l}\binom{j-1}{l}\binom{(m+k-1)-2-2l}{(2m-g+k-1)-2-l}\\
& =& [x^{2m-g+k-3}]\left(\sum_{l=0}^{j-1}(-1)^{l}\binom{j-1}{l}x^l(1+x)^{m+k-3-2l}\right)\\
& = & [x^{2m-g+k-3}]\left(
(1+x)^{m+k-3}\left(1-\frac{x}{(1+x)^2}\right)^{j-1}
\right)\\
& = & [x^{2m-g+k-3}]\left(
(1+x)^{m+k-3}\left(\frac{x^2+x+1}{(1+x)^2}\right)^{j-1}
\right).
\end{eqnarray*}
Therefore,
\begin{eqnarray*}
& & \sum_{S\in\cC(m,k,g)}\#\left(PF(S)\cap\left[1,\left\lfloor\frac{m+k-1}{2}\right\rfloor\right]\right) \\
& \leq & \sum_{j=1}^{\lfloor\frac{m+k-1}{2}\rfloor}\sum_{l=0}^{j-1}(-1)^{l}\binom{j-1}{l}\binom{(m+k-1)-2-2l}{(2m-g+k-1)-2-l}\\
& = & [x^{2m-g+k-3}]\left(
\sum_{j=1}^{\lfloor\frac{m+k-1}{2}\rfloor}(1+x)^{m+k-3}\left(\frac{x^2+x+1}{(1+x)^2}\right)^{j-1}
\right).
\end{eqnarray*}
Noting that
\begin{eqnarray*}
    & & \sum_{j=1}^{\lfloor\frac{m+k-1}{2}\rfloor}(1+x)^{m+k-3}\left(\frac{x^2+x+1}{(1+x)^2}\right)^{j-1}\\
    & = &  \frac{(1+x)^{m+k-1}}{(1+x)^{2}}\frac{\left(1-\left(\frac{x^2+x+1}{(1+x)^2}\right)^{\lfloor\frac{m+k-1}{2}\rfloor}\right)}{\left(1-\left(\frac{x^2+x+1}{(1+x)^2}\right)\right)}\\
    &=& \frac{(1+x)^{m+k-1}-(1+x+x^2)^{\lfloor\frac{m+k-1}{2}\rfloor}(1+x)^{(m+k-1)-2\lfloor\frac{m+k-1}{2}\rfloor}}{x}
\end{eqnarray*}
completes the proof of \eqref{eqn:t2_sum_cmkg_small}.
\end{proof}

\begin{lemma}\label{ave t2 0}
We have
\[
\lim_{g\to\infty}\frac{1}{g}\E_g[t_2]=0.
\]
\end{lemma}
Proposition \ref{estimate t_1 and t_2}(2) follows directly from this result together with Lemma \ref{Lemma Average to 0}.  Therefore, proving this result completes the proof of Theorem \ref{main result}(2).
\begin{proof}
Fix $\epsilon>0$. Consider the $M(\epsilon)$ given by Proposition \ref{F 2m}. For any $g$ we have
\begin{eqnarray*}
    \sum_{S\in\SS_g}t_2(S)&=& \sum_{S\in\B(g)}t_2(S)+\sum_{k=1}^{M(\epsilon)}\sum_{S\in\cC(k,g)}t_2(S)+\sum_{\substack{S\in\SS_g\\ F(S)> 2m(S)+M(\epsilon)}}t_2(S)\\
    &\leq& 2F_{g+4}+\sum_{k=1}^{M(\epsilon)}2F_{g+k+3}+\sum_{\substack{S\in\SS_g\\ F(S)> 2m(S)+M(\epsilon)}}(g+1)
\end{eqnarray*}
Noting that $F_n < \frac{\varphi^n+1}{\sqrt{5}}$ and applying Proposition \ref{F 2m} shows that this is less than
\begin{eqnarray*}
    & & 2\frac{\phi^{g+4}+1}{\sqrt{5}}+2\sum_{k=1}^{M(\epsilon)}\frac{\phi^{g+k+3}+1}{\sqrt{5}}+(g+1)\epsilon N(g)\\
    &=&\phi^{g+4} \frac{2}{\sqrt{5}}\left(1+\frac{\phi^{M(\epsilon)}-1}{\phi-1}\right)+\frac{2}{\sqrt{5}}(M(\epsilon)+1)+(g+1)\epsilon N(g).
\end{eqnarray*}

Therefore,
\[
\limsup_{g\to\infty}\frac{1}{g}\frac{1}{N(g)}\sum_{S\in\SS_g}t_2(S)\leq \epsilon.
\]
Since $\epsilon$ was arbitrary, we see that
\[
\lim_{g\to\infty}\frac{1}{g}\frac{1}{N(g)}\sum_{S\in\SS_g}t_2(S)=0.\qedhere
\]
\end{proof}

\section{The probability that a subset is contained in a semigroup of genus $g$}\label{sec:ind_prob}
In this section, we consider a set of positive integers and study the proportion of semigroups in $\SS_g$ that contain some of them but not others.  Intuitively, integers that are large relative to $g$ should be contained in most, or all, semigroups in $\SS_g$, and integers that are small relative to $g$ should be contained in very few semigroups in $\SS_g$.  It is the integers in `the middle' where the statistical behavior is not so obvious.  Our goal is to make this kind of reasoning precise.

Define a step function $f_1:[0,2]\setminus\{\gamma,2\gamma\}\to [0,1]$ by
\[
f_1(x) =
\left\{
	\begin{array}{ll}
		0  & \mbox{if } 0 \leq x <\gamma \\
		\frac{\sqrt{5}-1}{2}  & \mbox{if } \gamma< x<2\gamma \\
		1 & \mbox{if } 2\gamma<x\leq 2.
	\end{array}
\right.
\]
We show that $f_1\left(\frac{n}{g}\right)$ is a good approximation for the probability that $n$ is contained in a random element of $\SS_g$, and also that these probabilities for various $n$ are independent.

\begin{thm}\label{prob C in S and C' not in S}
Fix $l_1,l_2\geq 0$ and $\epsilon_1,\epsilon_2>0$. There exists an $M(\epsilon_1,\epsilon_2)>0$, such that for all $g>M(\epsilon_1,\epsilon_2)$ and all pairs of subsets
\[
C,C'\subseteq \big[1,(\gamma-\epsilon_1)g\big) \cup \big((\gamma+\epsilon_1)g,(2\gamma-\epsilon_1)g\big) \cup \big((2\gamma+\epsilon_1)g,2g\big]
\]
with $|C|=l_1$, $|C'|=l_2$ and $C\cap C'=\emptyset$, we have
\[
\left|\P_g[C\subseteq S \text{ and } C'\cap S=\emptyset]-\prod_{n\in C}f_{1}\left(\frac{n}{g}\right)\prod_{n\in C'}\left(1-f_{1}\left(\frac{n}{g}\right)\right)\right|<\epsilon_2.
\]
\end{thm}
In the next section we apply two special cases of this result.
\begin{cor}\label{cor:prob subset}
Fix $\epsilon_1, \epsilon_2 > 0$.
\begin{enumerate}
\item There exists an $M(\epsilon_1, \epsilon_2)>0$ such that for all $g>M(\epsilon_1, \epsilon_2)$ and
\[
n \in \big[1,(\gamma-\epsilon_1)g\big) \cup \big((\gamma+\epsilon_1)g,(2\gamma-\epsilon_1)g\big) \cup \big((2\gamma+\epsilon_1)g,2g\big]
\]
we have
\[
\left|\P_g[n \in S]-f_{1}\left(\frac{n}{g}\right)\right|<\epsilon_2.
\]

\item There exists an $M(\epsilon_1, \epsilon_2)>0$ such that for all $g>M(\epsilon_1, \epsilon_2)$ and
\[
i,j \in \big[1,(\gamma-\epsilon_1)g\big) \cup \big((\gamma+\epsilon_1)g,(2\gamma-\epsilon_1)g\big) \cup \big((2\gamma+\epsilon_1)g,2g\big]
\]
with $i\neq j$, we have
\[
\left|\P_g[\{i,j\}\cap S=\emptyset]-\left(1-f_{1}\left(\frac{i}{g}\right)\right)\left(1-f_{1}\left(\frac{j}{g}\right)\right)\right|<\epsilon_2.
\]
\end{enumerate}
\end{cor}
The second statement may be surprising because in general whether two positive integers $i,j$ are contained in a numerical semigroup $S$ are \emph{not} independent events.  For example, if $i \mid j$ then $i\in S$ implies $j \in S$.  The point here is that if $i$ is `too small', $\P_g[i \in S]$ is very small, and if neither of $i,j$ is small, then either they are in $S$ with very high probability, or they are too close together for the condition $i \in S$ to influence $\P_g[j \in S]$ in a meaningful way.

We start the proof of Theorem \ref{prob C in S and C' not in S} by recalling a result of Zhai that was conjectured by Zhao \cite[Conjecture 2]{Zhao}. Recall that the set $A_k$ is defined in Section \ref{sec:F3m}.
\begin{thm}\cite[Theorem 3.11]{Zhai}\label{Zhai summation of c}
The sum
\[
c=\frac{\phi}{\sqrt{5}}+\frac{1}{\sqrt{5}}\sum_{k=1}^{\infty}\sum_{A\in A_k}\phi^{|A|-|(A+A)\cap[0,k]|-k-1}
\]
converges. Moreover
\[
N(g)=c\phi^g+o(\phi^g).
\]
\end{thm}

We next need a technical result about sums of binomial coefficients.  The proof is similar to the proof of \cite[Proposition 7]{KaplanYe}, so we do not include it here.
\begin{lemma}\label{lemma:binom}
Fix $l_1,l_2\in\Z_{\ge 0}$ and $\epsilon>0$. As $g\to\infty$, we have
\[
\sum_{m= \lceil(\gamma-\epsilon)g\rceil}^{\lfloor(\gamma+\epsilon)g\rfloor}\binom{m-1-l_1}{g-m+1-l_2}=\frac{1}{\sqrt{5}}\phi^{g-l_1-l_2+1}+o(\phi^g).
\]
\end{lemma}

A main step in the proof of Theorem \ref{prob C in S and C' not in S} is to study elements between the expected size of the multiplicity and the expected size of the Frobenius number of a random $S\in \SS_g$.
\begin{lemma}\label{phi -l}
Fix $l\geq 0$ and $\epsilon_1,\epsilon_2>0$. There is an $M(\epsilon_1,\epsilon_2)>0$ such that for all $g>M(\epsilon_1,\epsilon_2)$ and all subsets $C\subseteq \big((\gamma+\epsilon_1)g,(2\gamma-\epsilon_1)g\big)$ with $|C|=l$, we have
\[
\left|\P_g[C\subseteq S]-\phi^{-l}\right|<\epsilon_2.
\]
\end{lemma}
\begin{proof}
We separately prove that for sufficiently large $g$, we have $\P_g[C\subseteq S]<\phi^{-l} + \epsilon_2$ and $\P_g[C\subseteq S]> \phi^{-l} - \epsilon_2$.  We first phrase this in terms of the sets $\B(g,m)$ and $A_k$ from~Section~\ref{sec:F3m}.

Fix $m\in \big((\gamma-\frac{\epsilon_1}{2})g,(\gamma+\frac{\epsilon_1}{2})g\big)$, so $m<\min(C)\leq \max(C)<2m$.  Proposition \ref{Zhao F<2m} implies~that
\[
\#\{S\in\B(g,m)\mid C\subseteq S\}=\binom{m-1-l}{2m-g-2-l}=\binom{m-1-l}{g-m+1}.
\]
Next, for $k\leq \min(\frac{g}{3},\frac{\epsilon_1}{2}g)$ and $A\in A_k$ we have $m+k<\min(C)\leq \max(C)<2m$ and
\begin{equation*}
    \begin{split}
    \#\{S\in \cC(m,k,A,g)\mid C\subseteq S\} &=\binom{m-1-|(A+A)\cap [0,k]|-l}{2m-g+k-|A|-|(A+A)\cap [0,k]|-l}\\ &=\binom{m-1-|(A+A)\cap [0,k]|-l}{g-m+|A|-k-1}.
    \end{split}
\end{equation*}
We start with the lower bound for $\P_g[C\subseteq S]$.
By Theorem \ref{Zhai summation of c}, we can choose $M$ sufficiently large so that
\[
\frac{\phi}{\sqrt{5}}+\frac{1}{\sqrt{5}}\sum_{k=1}^{M}\sum_{A\in A_k}\phi^{|A|-|(A+A)\cap[0,k]|-k-1} >c\left(1-\frac{\epsilon_2}{2}\phi^{-l}\right).
\]
For $g>\max(3M,\frac{2}{\epsilon_1}M)$, we have
\begin{eqnarray*}
\P_g[C\subseteq S]
        &\geq  & \frac{1}{N(g)}\sum_{m= \lceil(\gamma-\frac{\epsilon_1}{2})g\rceil}^{\lfloor(\gamma+\frac{\epsilon_1}{2})g\rfloor}\binom{m-1-l}{g-m+1} \\
& &+\frac{1}{N(g)}\sum_{k=1}^{M}\sum_{A\in A_k}\sum_{m= \lceil(\gamma-\frac{\epsilon_1}{2})g\rceil}^{\lfloor(\gamma+\frac{\epsilon_1}{2})g\rfloor}\binom{m-1-|(A+A)\cap [0,k]|-l}{g-m+|A|-k-1}.
\end{eqnarray*}
Applying Lemma \ref{lemma:binom} with $l_1 = l$ and $l_2 = 0$ shows that
\[
\sum_{m= \lceil(\gamma-\frac{\epsilon_1}{2})g\rceil}^{\lfloor(\gamma+\frac{\epsilon_1}{2})g\rfloor}\binom{m-1-l}{g-m+1} 
=\frac{1}{\sqrt{5}}\phi^{g+1-l}+o(\phi^g).
\]
Applying Lemma \ref{lemma:binom} with $l_1 = |(A+A)\cap [0,k]| + l$ and $l_2 = k+2-|A|$ shows that
\[
\sum_{m= \lceil(\gamma-\frac{\epsilon_1}{2})g\rceil}^{\lfloor(\gamma+\frac{\epsilon_1}{2})g\rfloor}\binom{m-1-|(A+A)\cap [0,k]|-l}{g-m+|A|-k-1}
=
\frac{1}{\sqrt{5}}\phi^{g-l+|A|-|(A+A)\cap[0,k]|-k-1}+o(\phi^g).
\]
We now have
\begin{eqnarray*}
\P_g[C\subseteq S] &\ge&
\phi^{-l}\frac{\phi^g}{N(g)}
       \left(\frac{\phi}{\sqrt{5}}+\frac{1}{\sqrt{5}}\sum_{k=1}^{M}\sum_{A\in A_k}\phi^{|A|-|(A+A)\cap[0,k]|-k-1}\right) +o(1)\\
    &>& \phi^{-l} \left(\frac{1}{c}+o(1)\right)c\left(1-\frac{\epsilon_2}{2}\phi^{-l}\right)+o(1)\\
    &=&\phi^{-l}-\frac{\epsilon_2}{2}+o(1).   
\end{eqnarray*}
Therefore, for sufficiently large $g$, all subsets $C\subseteq \big((\gamma+\epsilon_1)g,(2\gamma-\epsilon_1)g\big)$ with $|C|=l$ satisfy $\P_g[C\subseteq S]>\phi^{-l}-\epsilon_2$.

We now turn to the upper bound for $\P_g[C\subseteq S]$.  For $g>\max(3M,\frac{2}{\epsilon_1}M)$, we have
\begin{eqnarray*}
    \P_g[C\subseteq S]
    & \leq & \frac{1}{N(g)}\sum_{m= \lceil(\gamma-\frac{\epsilon_1}{2})g\rceil}^{\lfloor(\gamma+\frac{\epsilon_1}{2})g\rfloor}\binom{m-1-l}{g-m+1}\\
    && +\frac{1}{N(g)}\sum_{k=1}^{M}\sum_{A\in A_k}\sum_{m= \lceil(\gamma-\frac{\epsilon_1}{2})g\rceil}^{\lfloor(\gamma+\frac{\epsilon_1}{2})g\rfloor}\binom{m-1-|(A+A)\cap [0,k]|-l}{g-m+|A|-k-1}\\
    &&+\P_g\left[|m(S)-\gamma g|>\frac{\epsilon_1}{2}g\right]
    +\P_g\left[|F(S)-2m(S)|>M\right].
\end{eqnarray*}
As we analyzed above, the sum of the first two terms is
\begin{eqnarray*}
\phi^{-l}\frac{\phi^g}{N(g)}
       \left(\frac{\phi}{\sqrt{5}}+\frac{1}{\sqrt{5}}\sum_{k=1}^{M}\sum_{A\in A_k}\phi^{|A|-|(A+A)\cap[0,k]|-k-1}\right) +o(1)
       & < & \phi^{-l}\left(\frac{1}{c}+o(1)\right)c+o(1)\\ 
       & =& \phi^{-l}+o(1).
\end{eqnarray*}
Applying Theorem \ref{Kaplan Mult} and Proposition \ref{F 2m} shows that for all sufficiently large $g$, we have 
\[
\P_g\left[|m(S)-\gamma g|>\frac{\epsilon_1}{2}g\right]
    +\P_g\left[|F(S)-2m(S)|>M\right] \le \frac{\epsilon_2}{2}.
\]
We conclude that for sufficiently large $g$, all subsets $C\subseteq \big((\gamma+\epsilon_1)g,(2\gamma-\epsilon_1)g\big)$ with $|C|=l$ satisfy $\P_g[C\subseteq S]<\phi^{-l}+\epsilon_2$.
\end{proof}

We now prove a result about the probability that a set is contained in a random~$S\in \SS_g$.
\begin{thm}\label{prob C subset S}
Fix $l\geq 0$ and $\epsilon_1, \epsilon_2 > 0$. There exists an $M(\epsilon_1, \epsilon_2)>0$ such that for all $g>M(\epsilon_1, \epsilon_2)$ and all subsets
\[
C\subseteq \big[1,(\gamma-\epsilon_1)g\big) \cup \big((\gamma+\epsilon_1)g,(2\gamma-\epsilon_1)g\big) \cup \big((2\gamma+\epsilon_1)g,2g\big]
\]
of size $|C|=l$, we have
\[
\left|\P_g[C\subseteq S]-\prod_{n\in C}f_{1}\left(\frac{n}{g}\right)\right|<\epsilon_2.
\]
\end{thm}
\begin{proof}
By Theorem $\ref{Kaplan Mult}$, there is $M_1(\epsilon_1,\epsilon_2)$ such that $g> M_1(\epsilon_1,\epsilon_2)$ implies that
\begin{eqnarray*}
\P_g[m(S)  \leq (\gamma-\epsilon_1)g] &<& \epsilon_2, \\
\P_g[F(S) \geq  (2\gamma+\epsilon_1)g] &<& \frac{\epsilon_2}{2}.
\end{eqnarray*}
Lemma \ref{phi -l} implies that there is $M_2(\epsilon_1, \epsilon_2) >0$ such that for all $g>M_2(\epsilon_1, \epsilon_2)$ and all subsets
\[
C'\subseteq \big((\gamma+\epsilon_1)g,(2\gamma-\epsilon_1)g\big),
\]
of size $|C'|\leq l$, we have
\[
\left|\P_g[C'\subseteq S]-\phi^{-|C'|}\right|<\frac{\epsilon_2}{2}.
\]
Let $M=\max(M_1(\epsilon_1,\epsilon_2),M_2(\epsilon_1, \epsilon_2))$. Pick $g>M$ and a subset
\[
C\subseteq \big[1,(\gamma-\epsilon_1)g\big) \cup \big((\gamma+\epsilon_1)g,(2\gamma-\epsilon_1)g\big) \cup \big((2\gamma+\epsilon_1)g,2g\big]
\]
of size $|C|=l$. We split $C$ into three parts: 
\begin{eqnarray*}
C_1 & = & C\cap \big[1,(\gamma-\epsilon_1)g\big),\\ 
C_2 &= & C\cap \big((\gamma+\epsilon_1)g,(2\gamma-\epsilon_1)g\big),\\ 
C_3 &= & C\cap \big((2\gamma+\epsilon_1)g,2g\big].
\end{eqnarray*}
We consider two cases.
\begin{itemize}
    \item Case 1: $C_1\neq\emptyset$. Then $\prod_{n\in C}f_{1}\left(\frac{n}{g}\right)=0$. Moreover,
    \[
    \P_g[C\subseteq S]\leq \P_g[m(S)<(\gamma-\epsilon_1)g]<\epsilon_2.
    \]
    
    \item Case 2: $C_1=\emptyset$. Suppose $|C_2|=l_1\leq l$, so $\prod_{n\in C}f_{1}\left(\frac{n}{g}\right)=\phi^{-l_1}$. Now
    \[
    \P_g[C\subseteq S]\leq \P_g[C_2\subseteq S]< \phi^{-l_1}+\frac{\epsilon_2}{2}.
    \]
We have
    \[
    \P_g[C\subseteq S] \geq \P_g[C_2\subseteq S]-\P_g[C_3 \not\subset S].
    \]
Now if $C_3=\emptyset$, then $\P_g[C_3\not\subset S]=0$ and if $C_3\neq\emptyset$, then
$$\P_g[C_3\not\subset S]\leq P_g\Big[F(S)\geq (2\gamma+\epsilon_1)g\Big]<\frac{\epsilon_2}{2}.$$
Therefore, we have 
    \[
    \P_g[C\subseteq S]  >\phi^{-l_1}-\frac{\epsilon_2}{2}-\frac{\epsilon_2}{2}.
    \]
    We conclude that
    \[
    \left|\P_g[C\subseteq S]-\phi^{-l_1}\right|<\epsilon_2.\qedhere
    \]
\end{itemize}
\end{proof}

We now have all the tools to prove the main result of this section.
\begin{proof}[Proof of Theorem \ref{prob C in S and C' not in S}]
By Theorem \ref{prob C subset S}, there exists an $M(\epsilon_1,\epsilon_2)>0$ such that for all $g>M(\epsilon_1,\epsilon_2)$ and all subsets
\[
D\subseteq \big[1,(\gamma-\epsilon_1)g\big) \cup \big((\gamma+\epsilon_1)g,(2\gamma-\epsilon_1)g\big) \cup \big((2\gamma+\epsilon_1)g,2g\big]
\]
of size $|D|\leq l_1+l_2$, we have
\[
\left|\P_g[D\subseteq S]-\prod_{n\in D}f_{1}\left(\frac{n}{g}\right)\right|<\frac{\epsilon_2}{2^{l_2}}.
\]
Consider $g>M(\epsilon_1,\epsilon_2)$ and a pair of subsets
\[
C,C'\subseteq \big[1,(\gamma-\epsilon_1)g\big) \cup \big((\gamma+\epsilon_1)g,(2\gamma-\epsilon_1)g\big) \cup \big((2\gamma+\epsilon_1)g,2g\big]
\]
with $|C|=l_1,\ |C'|=l_2$, and $C\cap C'=\emptyset$.  By inclusion-exclusion, we know that
\[
\P_g[C\subseteq S \text{ and } C'\cap S=\emptyset] =\sum_{B\subseteq C'}(-1)^{|B|} \P_g[(C\cup B)\subseteq S].
\]
We see that
\begin{eqnarray*}
  &  &\left|\P_g[C\subseteq S \text{ and } C'\cap S=\emptyset]
    -\prod_{n\in C}f_{1}\left(\frac{n}{g}\right)\prod_{n\in C'}\left(1-f_{1}\left(\frac{n}{g}\right)\right)\right|\\
    &= &\left|\sum_{B\subseteq C'}(-1)^{|B|} \P_g[(C\cup B)\subseteq S]
    -\sum_{B\subseteq C'}(-1)^{|B|}\prod_{n\in C\cup B}f_{1}\left(\frac{n}{g}\right)\right|\\
    &\leq & \sum_{B\subseteq C'} \left|\P_g[(C\cup B)\subseteq S]-\prod_{n\in C\cup B}f_{1}\left(\frac{n}{g}\right)\right|\\
    &< & \sum_{B\subseteq C'}\frac{\epsilon_2}{2^{l_2}} =\epsilon_2,
\end{eqnarray*}
where in the last step we applied Theorem \ref{prob C subset S} with $D = C\cup B$. 
\end{proof}

\section{The weight of a typical numerical semigroup}\label{Sec: weight}

We first determine the expected value of the weight of a numerical semigroup of genus $g$.
\begin{proof}[Proof of Theorem \ref{Weight result}(1)]
Let $\alpha(S)=\sum_{x\in\Hs(S)}x$. So for $S\in S_g$ we have $w(S)=\alpha(S)-\frac{g(g+1)}{2}$.  We will prove that
\begin{equation}\label{avg_alpha}
\lim_{g\to\infty}\frac{1}{g^2}\E_g[\alpha(S)]=\frac{9+\sqrt{5}}{20}.
\end{equation}
Assuming this for now, noting that $\frac{1}{10\phi}=\frac{9+\sqrt{5}}{20}-\frac{1}{2}$ completes the proof.

Our goal is now to prove \eqref{avg_alpha}. Every numerical semigroup $S$ satisfies $F(S) \le 2g(S)-1$. By linearity of expectation, we know that
\begin{eqnarray*}
\E_g[\alpha(S)] & = & \sum_{n=1}^{2g-1}n\P_g[n\notin S] \\
    & = & \sum_{n=1}^{2g-1}n\left(1-f_1\left(\frac{n}{g}\right)\right)
    +\sum_{n=1}^{2g-1}n\left(f_1\left(\frac{n}{g}\right)-\P_g[n\in S]\right).
\end{eqnarray*}
We have
\[
\sum_{n=1}^{2g-1}n\left(1-f_1\left(\frac{n}{g}\right)\right)
=
\frac{(\gamma g)^2}{2}+(1-\phi^{-1})\left(\frac{(2\gamma g)^2}{2}-\frac{(\gamma g)^2}{2}\right)+O(g).
\]
Applying Corollary \ref{cor:prob subset}(1) gives an $M(\epsilon_1,\epsilon_2)$ such that for all $g> M(\epsilon_1,\epsilon_2)$, we have
\begin{equation*}
    \begin{split}
    \left|\sum_{n=1}^{2g-1}n\left(f_1\left(\frac{n}{g}\right)-\P_g[n\in S]\right)\right|
    &\leq 2g\sum_{n=1}^{2g-1}\left|f_1\left(\frac{n}{g}\right)-\P_g[n\in S]\right|\\
    &\leq (2g)(4\epsilon_1 g  +(2-4\epsilon_1)g \epsilon_2).
    \end{split}
\end{equation*}

Note that $\frac{\gamma^2}{2}+(1-\phi^{-1})\frac{3\gamma^2}{2}=\frac{9+\sqrt{5}}{20}$. Since $\epsilon_1,\epsilon_2$ were arbitrary, we conclude that
\[
\E_g[\alpha]=\frac{9+\sqrt{5}}{20}g^2+o(g^2),
\]
completing the proof of \eqref{avg_alpha}.
\end{proof}

\begin{thm}\label{average w square}
We have
\[
\lim_{g\to\infty}\frac{1}{g^4}\E_g[w(S)^2]=\left(\frac{1}{10\phi}\right)^2.
\]
\end{thm}

Before giving the proof we use this result to complete the proof of Theorem~\ref{Weight result}(2).
\begin{proof}[Proof of Theorem \ref{Weight result}(2)]
Apply Lemma \ref{Concentration by Chebychev} together with Theorem \ref{Weight result}(1) and Theorem \ref{average w square}.
\end{proof}

\begin{proof}[Proof of Theorem \ref{average w square}]
We will show that
\begin{equation}\label{avg_alpha2}
\lim_{g\to\infty}\frac{1}{g^4}\E_g[\alpha(S)^2]=\left(\frac{9+\sqrt{5}}{20}\right)^2.
\end{equation}
Assuming this for now, we complete the proof.  Since $w(S)=\alpha(S)-\frac{g(g+1)}{2}$, we see that
\begin{eqnarray*}
    \lim_{g\to\infty}\frac{1}{g^4}\E_g[w(S)^2] 
    &=& \lim_{g\to\infty}\frac{1}{g^4}\E_g[\alpha(S)^2]
    -\lim_{g\to\infty}\frac{g(g+1)}{g^4}\E_g[\alpha(S)]
    +\lim_{g\to\infty}\frac{1}{g^4}\frac{g^2(g+1)^2}{4}\\
    &=& \left(\frac{9+\sqrt{5}}{20}\right)^2-\frac{9+\sqrt{5}}{20}+\frac{1}{4}=\left(\frac{1}{10\phi}\right)^2,
\end{eqnarray*}
where we used the expressions in \eqref{avg_alpha} and \eqref{avg_alpha2}.  Therefore, we only need to prove \eqref{avg_alpha2}.

For $1\leq i\leq 2g-1$, consider the following random variables on $\SS_g$:
\[
\psi_i(S)=\begin{cases}1 &\text{ if } i\notin S\\ 0 &\text{ if } i\in S\end{cases}.
\]
Therefore, $\alpha(S)=\sum_{i=1}^{2g-1}i\psi_i(S).$ By linearity of expectation we have
\begin{eqnarray*}
\E_g[\alpha(S)^{2}] &= &\sum_{i=1}^{2g-1}\sum_{j=1}^{2g-1}ij\E_g[\psi_i\psi_j] =\sum_{i=1}^{2g-1}\sum_{j=1}^{2g-1}ij\P_g[\{i,j\}\cap S=\emptyset]\\
& = & 
\sum_{i=1}^{2g-1}\sum_{j=1}^{2g-1}ij\left(1-f_{1}\left(\frac{i}{g}\right)\right)\left(1-f_{1}\left(\frac{j}{g}\right)\right) \\
& & -\sum_{i=1}^{2g-1}\sum_{j=1}^{2g-1}ij 
    \left(
    \left(1-f_{1}\left(\frac{i}{g}\right)\right)\left(1-f_{1}\left(\frac{j}{g}\right)\right)
    -\P_g[\{i,j\}\cap S=\emptyset]
    \right).
\end{eqnarray*}
We estimate the size of each term separately, starting with the first.  We have
\[
    \sum_{i=1}^{2g-1}\sum_{j=1}^{2g-1}ij\left(1-f_{1}\left(\frac{i}{g}\right)\right)\left(1-f_{1}\left(\frac{j}{g}\right)\right)
    =\left(\sum_{i=1}^{2g-1}i\left(1-f_{1}\left(\frac{i}{g}\right)\right)\right)^2.
\]
Note that,
\[
\sum_{i=1}^{2g-1}i\left(1-f_{1}\left(\frac{i}{g}\right)\right)
=
\frac{\gamma^2}{2}g^2+ (1-\phi^{-1})\left(\frac{4\gamma^2}{2}-\frac{\gamma^2}{2}\right)g^2 +O(g).
\]
Therefore,
\[
\sum_{i=1}^{2g-1}\sum_{j=1}^{2g-1}ij\left(1-f_{1}\left(\frac{i}{g}\right)\right)\left(1-f_{1}\left(\frac{j}{g}\right)\right)
=
\left(\frac{9+\sqrt{5}}{20}\right)^2 g^4+O(g^3).
\]

Now we estimate the second term. It is clear that 
\begin{eqnarray*}
    & &\left|\sum_{i=1}^{2g-1}\sum_{j=1}^{2g-1}ij 
    \left(
    \left(1-f_{1}\left(\frac{i}{g}\right)\right)\left(1-f_{1}\left(\frac{j}{g}\right)\right)
    -\P_g[\{i,j\}\cap S=\emptyset]
    \right) \right|\\
    &\leq & \sum_{i=1}^{2g-1}\sum_{j=1}^{2g-1}ij 
    \left|    \left(1-f_{1}\left(\frac{i}{g}\right)\right)\left(1-f_{1}\left(\frac{j}{g}\right)\right)
    -\P_g[\{i,j\}\cap S=\emptyset]
    \right|.
\end{eqnarray*}
We first consider the terms with $i=j$, and see that
\[
\sum_{i=1}^{2g-1}i^2 \left|    \left(1-f_{1}\left(\frac{i}{g}\right)\right)^2
    -\P_g[\{i\}\cap S=\emptyset]
    \right|\leq \sum_{i=1}^{2g-1}i^2=O(g^3).
\]

Next we consider the terms with $i\neq j$. Fix $\epsilon_1,\epsilon_2>0$. Corollary \ref{cor:prob subset}(2) gives an $M_2(\epsilon_1,\epsilon_2)$ such that for all $g > M_2(\epsilon_1,\epsilon_2)$ and all $i\neq j$ with 
\[
\{i,j\}\subseteq \big[1,(\gamma-\epsilon_1)g\big) \cup \big((\gamma+\epsilon_1)g,(2\gamma-\epsilon_1)g\big) \cup \big((2\gamma+\epsilon_1)g,2g\big]
\]
we have 
\[
\left|\P_g[\{i,j\}\cap S=\emptyset]-\left(1-f_{1}\left(\frac{i}{g}\right)\right)\left(1-f_{1}\left(\frac{j}{g}\right)\right)\right|<\epsilon_2.
\]
Therefore,  
\begin{eqnarray*}
& & \sum_{i=1}^{2g-1}\sum_{\substack{j=1\\ j\neq i}}^{2g-1}ij 
    \left|    \left(1-f_{1}\left(\frac{i}{g}\right)\right)\left(1-f_{1}\left(\frac{j}{g}\right)\right)
    -\P_g[\{i,j\}\cap S=\emptyset]
    \right| \\
& \le & 
\sum_{i=1}^{2g-1}\sum_{j=1}^{2g-1}ij \epsilon_2 +
\sum_{i\in \big((\gamma-\epsilon_1)g,(\gamma+\epsilon_1)g\big)} \sum_{j=1}^{2g-1}ij 
+ \sum_{i\in \big((2\gamma-\epsilon_1)g,(2\gamma+\epsilon_1)g\big)} \sum_{j=1}^{2g-1}ij \\
& &+
\sum_{i =1}^{2g-1}\sum_{j\in \big((\gamma-\epsilon_1)g,(\gamma+\epsilon_1)g\big)} ij
+ 
\sum_{i =1}^{2g-1}\sum_{j\in \big((2\gamma-\epsilon_1)g,(2\gamma+\epsilon_1)g\big)} ij \\
& \le & 
\epsilon_2\frac{(2g)^2}{2}\frac{(2g)^2}{2}+ 4(2\epsilon_1 g)(2g) \frac{(2g)^2}{2}=\epsilon_2 4g^4+\epsilon_1 32g^4.
\end{eqnarray*}
Combining everything, since $\epsilon_1,\epsilon_2$ were arbitrary, we get
\[
\E_g[\alpha(S)^2] =\left(\frac{9+\sqrt{5}}{20}\right)^2 g^4+o(g^4).
\]
\end{proof}

\section{Counting Numerical Semigroups with Large Embedding Dimension}\label{Sec: large ED}

A main idea of this section is to construct a bijection between numerical semigroups with fixed multiplicity, genus, and embedding dimension and certain finite sequences of positive integers.  These sequences are the initial segments of Kunz coordinate vectors of numerical semigroups.  We recall some notation and basic facts about these objects.
\begin{defn}
Let $S$ be a numerical semigroup.  The Ap\'ery set of $S$ with respect to an element $m \in S$ is
\[
\Ap(S;m) = \{s\in S\colon s - m \not\in S\}.
\]
\end{defn}
It is easy to see that there is one element of $\Ap(S;m)$ in each residue class modulo $m$.  We can write $\Ap(S;m) = \{0,a_1,a_2,\ldots, a_{m-1}\}$ where each $a_i \equiv i \pmod{m}$.  We then define the nonnegative integer $k_i$ by $a_i = k_i m + i$.  Note that if $m = m(S)$, then each $k_i \ge 1$.  
\begin{defn}
The Kunz coordination vector of $S$ with respect to $m$ is $(k_1,\ldots, k_{m-1})$.  Let $\KV_m$ denote the function that takes a numerical semigroup containing $m$ to its Kunz coordinate vector with respect to $m$.
\end{defn}
We collect some results about Kunz coordinate vectors of numerical semigroups.
\begin{thm}\label{Kunz_bijection}\cite{kunz, RGGB}
The map $\KV_m$ gives a bijection between $S\in \SS_g$ with $m(S) = m$ and $(x_1,\ldots, x_{m-1}) \in \Z_{\ge 1}^{m-1}$ satsifying:
\begin{eqnarray*}
x_i + x_j \ge x_{i+j}, & \text{for all } &1 \le i \le j \le m-1 \text{ with } i+j < m,\\
x_i + x_j + 1\ge x_{i+j-m}, & \text{for all } & 1 \le i \le j \le m-1 \text{ with } i+j > m,\\
 \sum_{i=1}^{m-1} x_i  = g. & &
\end{eqnarray*}
\end{thm}
\begin{proposition}
If $S$ is a numerical semigroup with $m=m(S)$, then $\A(S) \setminus \{m\} \subseteq\Ap(S;m)$. More precisely, if $\KV_m(S) = (k_1,\ldots, k_{m-1})$ we have
\begin{equation*}
    \begin{split}
    \A(S)=\{m\}\cup \{&mk_i+i\mid \not\exists j_1,j_2\in[1,m-1]\colon j_1+j_2=i,\ k_{j_1}+k_{j_2}=k_i,\\ 
    &\text{ and } \not\exists j_1,j_2\in[1,m-1]\colon j_1+j_2=m+i,\ k_{j_1}+k_{j_2}+1=k_i\} .   
    \end{split}
\end{equation*}
\end{proposition}
\noindent For a more detailed discussion of this material, see \cite[Section 4]{KaplanONeill}.

In order to state the first main result of this section, we introduce some notation.  Suppose $\overline{x} = (x_1,\ldots, x_t) \in \{1,2,3\}^t$.  We define
\begin{enumerate}
\item $a(\overline{x})=\#\{i\in [1,t]\mid x_i=2\}$,
\item $b(\overline{x})=\#\{i\in [1,t]\mid x_i=3\}$,
\item $c(\overline{x})=\#\{i\in [1,t]\mid \exists j_1,j_2\in [1,t]:\; j_1+j_2=i, (x_{j_1},x_{j_2},x_{i})=(1,1,2)\}$.
\end{enumerate}

Theorem \ref{Kunz_bijection} gives a bijection between numerical semigroups $S$ with $m(S) = m$ and $g(S) = g$ and a certain set of integer tuples of length $m-1$.  We consider a refined version of a this result that applies in the case where $g(S)$ and $e(S)$ are not too far away from $m(S)$. 
\begin{thm}\label{thm:bijection}
Fix integers $k_1$ and $k_2$ satisfying $-1 \le k_1 \le k_2$ and $m \ge 2k_1+2$.  For $(x_1,\ldots, x_{m-1}) \in \Z_{\ge 0}^{m-1}$, let $\overline{x} = (x_1,\ldots, x_{2k_1+1})$.  

There is a bijection between the set of numerical semigroups $S$ satisfying $m(S) = m$, $g(S) = m+k_1$, and $e(S) = g(S) - k_2$, and sequences $(x_1,\ldots, x_{m-1})$ satisfying:
\begin{enumerate}
    \item $x_1,\dots,x_{m-1} \in \{1,2,3\}$.
    \item If $i\geq 2k_1$, then $x_i\in\{1,2\}$.
    \item Whenever $i_1, i_2, i_3\in [1,2k_1+1]$ satisfy $i_1+i_2=i_3$, we have $(x_{i_1},x_{i_2},x_{i_3})\neq(1,1,3)$.
    \item $\#\{i\in [2k_1+2,m-1]\mid x_i=2\}=k_1+1-a(\overline{x})-2b(\overline{x})$.
    \item $a(\overline{x})+b(\overline{x})-c(\overline{x})=2k_1+1-k_2$.
\end{enumerate}
\end{thm}
\noindent Note that conditions (4) and (5) imply that if such an $S$ exists, then
\begin{itemize}
    \item $a(\overline{x})+2b(\overline{x})\leq k_1+1$.
    \item $k_2\leq 2k_1+1$.
\end{itemize}
In the course of proving Theorem \ref{thm:bijection}, we will express $e(S)$ in terms of $(x_1,\ldots, x_{m-1})$.  We highlight this result because we will apply it in the discussion that follows.
\begin{proposition}\label{embedding_overlinex}
Suppose $(x_1,\ldots, x_{m-1}) = \KV_m(S)$ where $S$ is a numerical semigroup satisfying $m(S) = m$, $g(S) = m+k_1$, and $m\geq 2k_1+2$.  Let $\overline{x} = (x_1,\ldots, x_{2k_1+1})$.  Then 
\[
e(S) = g-2k_1-1+a(\overline{x})+b(\overline{x})-c(\overline{x}).
\]
\end{proposition}

Before proving Theorem \ref{thm:bijection}, we show how to use it to prove Theorem \ref{large embedding dimension}.
\begin{proposition}\label{Count k_1,k_2}
Fix integers $k_1, k_2$ satisfying $-1 \le k_1 \le k_2$ and $g\geq 4k_1+3$.  Suppose $\overline{x}=(x_1,\dots,x_{2k_1+1}) \in \{1,2,3\}^{2k_1 +1}$ satisfies the following conditions:
\begin{enumerate}
    \item Whenever $i_1, i_2, i_3\in [1,2k_1+1]$ satisfy $i_1+i_2=i_3$, we have $(x_{i_1},x_{i_2},x_{i_3})\neq(1,1,3)$.
    \item $a(\overline{x})+b(\overline{x})-c(\overline{x})=2k_1+1-k_2$.
    \item $a(\overline{x})+2b(\overline{x})\leq k_1+1$.
\end{enumerate}
The number of numerical semigroups $S$ for which $g(S)=g,\ m(S)=g-k_1,\ e(S)=g-k_2$, and the first $2k_1+1$ coordinates of $\KV_m(S)$ are given by $\overline{x}$ is $\binom{g-3k_1-2}{k_1+1-a(\overline{x})-2b(\overline{x})}.$
\end{proposition}
\begin{proof}
Suppose $S$ satisfies $g(S)=g,\ m(S) = m = g-k_1,\ e(S) = g-k_2$, and $\KV_m(S) = (x_1,\ldots, x_{2k_1+1}, k_{2k_1 +2},\ldots, k_{m-1})$.  Since $g(S) = g$, Theorem \ref{thm:bijection} implies that the number of $k_{2k_1 +2},\ldots, k_{m-1}$ equal to $2$ must be $k_1+1-a(\overline{x})-2b(\overline{x})$, and the rest of the elements must be equal to $1$.  Note that $(m-1)-(2k_1+1) = g-3k_1-2$.  Therefore, we have $\binom{g-3k_1-2}{k_1+1-a(\overline{x})-2b(\overline{x})}$ choices for $k_{2k_1 +2},\ldots, k_{m-1}$.  Theorem \ref{thm:bijection} says that each choice gives a semigroup satisfying the properties we are looking for, and that these are all such semigroups. Note that condition (3) ensures that $k+1-a(\overline{x})-2b(\overline{x})\geq 0$ and the condition $g\geq 4k_1+3$ ensures that $k_1+1-a(\overline{x})-2b(\overline{x})\leq g-3k_1-2.$
\end{proof}

Consider the collection of all sequences $\overline{x}$ satisfying the conditions of Theorem \ref{thm:bijection}, but now where we allow $k_2$ to vary.  This leads to the following definition.  
\begin{defn}
Fix $k \in \Z_{\ge 0}$. Let $\Y(k)$ be the collection of all tuples $\overline{x}=(x_1,\dots,x_{2k+1})\in\{1,2,3\}^{2k+1}$ satisfying the following conditions:
\begin{enumerate}
    \item Whenever $i_1, i_2, i_3\in [1,2k+1]$ satisfy $i_1+i_2=i_3$, we have $(x_{i_1},x_{i_2},x_{i_3})\neq(1,1,3)$. 
    \item $a(\overline{x})+2b(\overline{x})\leq k+1$.
\end{enumerate}
\end{defn}
\begin{thm}
Fix an integer $k\geq -1$. For $g\geq 4k+3$ we have
\[
\#\{S\in\S_g\mid m(S)=g-k\} =\sum_{\overline{x}\in\Y(k)}\binom{g-3k-2}{k+1-a(\overline{x})-2b(\overline{x})}.
\]
\end{thm}
\begin{proof}
For $\overline{x}\in\Y(k)$, let $k_2=2k+1-a(\overline{x})-b(\overline{x})+c(\overline{x})$. We have
\[
k_2=k-a(\overline{x})-b(\overline{x})+c(\overline{x})+k+1\geq k+b(\overline{x})+c(\overline{x})\geq k.
\]
We apply Proposition \ref{Count k_1,k_2} for each $\overline{x}\in\Y(k)$ with the corresponding $k_2$ and add the results.
\end{proof}

In the notation of Theorem \ref{thm:count_mlarge}, this means that for each $k\ge -1$ and $g \ge 4k+3$, we have
\[
\frac{1}{(k+1)!}f_k(x)=\sum_{\overline{x}\in\Y(k)}\binom{x-3k-2}{k+1-a(\overline{x})-2b(\overline{x})}.
\]
At the end of this paper we list the sets $\Y(k)$ for $-1\le k \le 2$.  A simple computation gives $f_{-1}(x),\ldots, f_2(x)$. We see that they match the formulas in \cite[Corollary 14]{Kaplan}.

We return to the problem of counting semigroups $S\in \SS_g$ with large embedding dimension.
\begin{thm}\label{thm:large_embedding_binomials}
Fix an integer $l\ge -1$. For $g\geq 4l+3$ we have
\[
\#\{S\in\S_g\mid e(S)=g-l\} =
\sum_{k=-1}^{l}\sum_{\substack{\overline{x}\in\Y(k) \\ a(\overline{x})+b(\overline{x})-c(\overline{x})=2k+1-l}}\binom{g-3k-2}{k+1-a(\overline{x})-2b(\overline{x})}.
\]
\end{thm}
\begin{proof}
We divide up the semigroups in $\S_g$ with $e(S) = g-l$ by multiplicity and see that
\[
\{S\in\S_g\mid e(S)=g-l\}=\bigcup_{k=-1}^{l} \{S\in\S_g\mid e(S)=g-l,m(S)=g-k\}.
\]
By Theorem \ref{thm:bijection}, if $S\in S_g$ satisfies $e(S)=g-l$ and $m(S)=g-k$, then $\overline{x}\in \Y(k)$ and $a(\overline{x})+b(\overline{x})-c(\overline{x})=2k+1-l$. By Proposition \ref{Count k_1,k_2}, the number of numerical semigroups corresponding to a given $\overline{x}$ is $\binom{g-3k-2}{k+1-a(\overline{x})-2b(\overline{x})}$. The result follows.
\end{proof}

The only remaining thing needed to complete the proof of Theorem \ref{large embedding dimension} is to establish some basic properties of the polynomials on the right-hand side of Theorem \ref{thm:large_embedding_binomials}. Define 
\[
H_l(x)=\sum_{k=-1}^{l}\sum_{\substack{\overline{x}\in\Y(k) \\ a(\overline{x})+b(\overline{x})-c(\overline{x})=2k+1-l}}\binom{x-3k-2}{k+1-a(\overline{x})-2b(\overline{x})}.
\]

\begin{proposition}\label{deg of h_l}
Fix an integer $l\geq -1$. Let $l_1=\lfloor\frac{l+1}{2}\rfloor$. Then $H_l(x)$ is a polynomial of degree $l_1$ and $l_1!H_l(x)$ is a monic polynomial with integer coefficients.
\end{proposition}
\noindent Proving this statement completes the proof of Theorem \ref{large embedding dimension}.
\begin{proof}
The degree of $H_l(x)$ is
\[
\max\{k+1-a(\overline{x})-2b(\overline{x})\mid -1\leq k\leq l,\ \overline{x}\in \Y(k),2k+1+c(\overline{x})-a(\overline{x})-b(\overline{x})=l\}.
\]
Suppose that $-1\leq k\leq l$ and $\overline{x}\in \Y(k)$ satisfies $2k+1+c(\overline{x})-a(\overline{x})-b(\overline{x})=l$. We have
\[
2k+1=l+(a(\overline{x})-c(\overline{x}))+b(\overline{x})\geq l.
\]
This means that $k\geq \frac{l-1}{2}$.
Next,
\[
k+1-a(\overline{x})-2b(\overline{x})=l-k-c(\overline{x})-b(\overline{x})\leq l-k\leq \frac{l+1}{2}.
\]
This implies $\deg(H_l(x))\leq \frac{l+1}{2}$. Since $\deg(H_l(x))$ is an integer, we see that $\deg(H_l(x))\leq l_1$. Note that $k+1-a(\overline{x})-2b(\overline{x})=l_1$ if and only if $b(\overline{x})=c(\overline{x})=0$ and $k=l-l_1$.
\begin{itemize}[leftmargin=*]
    \item If $l$ is odd, then $l=2l_1-1$. Take $k=l-l_1=l_1-1$. If $\overline{x}\in\Y(k)$ satisfies $b(\overline{x})=c(\overline{x})=0$, then
    \[
    l=2k+1-a(\overline{x})=2l_1-1+a(\overline{x})=l-a(\overline{x}).
    \]
    This implies $a(\overline{x})=0$. There is a unique such $\overline{x}$, which is $\overline{x}=(1,1,\dots,1)\in\Y(k)$.
    
    \item If $l$ is even, then $l=2l_1$. Take $k=l-l_1=l_1$.  If $\overline{x}\in\Y(k)$ satisfies $b(\overline{x})=c(\overline{x})=0$, then
    \[
    l=2k+1-a(\overline{x})=2l_1+1+a(\overline{x})=l+1-a(\overline{x}).
    \]
    This implies $a(\overline{x})=1$. There is a unique such $\overline{x}$, which is $\overline{x}=(2,1,\dots,1)\in\Y(k)$.
\end{itemize}
This completes the proof.
\end{proof}

\subsection{The proof of Theorem \ref{thm:bijection}.}

The goal of the rest of this section is to prove Theorem~\ref{thm:bijection}. 

We need several facts about the Kunz coordinate vector of a numerical semigroup with multiplicity $m$ and genus $g$.

\begin{lemma}\cite[Lemma 11]{Kaplan}\label{ki in 1,2,3}
Suppose $S$ is a numerical semigroup with $g(S)=g$, $m(S)=m$, and $\KV_m(S) = (x_1,\dots,x_{m-1})$. If $2g<3m+2$, then $\{k_1,\dots,k_{m-1}\}\subseteq \{1,2,3\}$.
\end{lemma}

\begin{lemma}\label{ki=3}
Suppose $S$ is a numerical semigroup with $g(S)=g$ and $m(S)=m$. Let $\KV_m(S) =(x_1,\dots,x_{m-1})$ If $i\in [1,m-1]$ satisfies $k_i=3$, then
\[
g\geq m+1+\left\lceil\frac{i-1}{2}\right\rceil.
\]
\end{lemma}
\begin{proof}
The set $[1,i-1]$ can be partitioned as a union of $\lceil\frac{i-1}{2}\rceil$ subsets of the form $\{j_1,j_2\}$ with $j_1+j_2= i$. For each such $\{j_1,j_2\}$, at least one of $x_{j_1},x_{j_2}$ must be at least $2$. Therefore,
\[
g-(m-1)=\sum_{j=1}^{m-1}(x_j-1)\geq 2+\left\lceil\frac{i-1}{2}\right\rceil.\qedhere
\]
\end{proof}

\begin{lemma}\label{ki=2 is gen}
Suppose $S$ is a numerical semigroup with $g(S)=g$ and $m(S)=m$. Let $\KV_m(S) = (x_1,\dots,x_{m-1})$. If $i\in [1,m-1]$ satisfies $k_i=2$ and $mk_i+i\in\A(S)$, then
\[
g\geq m+\left\lceil\frac{i-1}{2}\right\rceil.
\]
\end{lemma}
\begin{proof}
The set $[1,i-1]$ can be partitioned as a union of $\lceil\frac{i-1}{2}\rceil$ subsets of the form $\{j_1,j_2\}$ with $j_1+j_2= i$. For each such $\{j_1,j_2\}$, at least one of $x_{j_1},x_{j_2}$ must be at least $2$. Therefore,
\[
g-(m-1)=\sum_{j=1}^{m-1}(x_j-1)\geq 1+\left\lceil\frac{i-1}{2}\right\rceil.\qedhere
\]
\end{proof}

\begin{lemma}\label{ki=3 is gen}
Suppose $S$ is a numerical semigroup with $g(S)=g$ and $m(S)=m$. Let $\KV_m(S) =(x_1,\dots,x_{m-1})$. If $3m+i \in\A(S)$ for some $i \in [1,m-1]$, then $g\geq \frac{3m}{2}$.
\end{lemma}
\begin{proof}
The set $[1,m-1]\setminus\{i\}$ can be partitioned as a union of subsets of the form $\{j_1,j_2\}$ with $j_1+j_2\equiv i\pmod{m}$.
Since $3m+ i \in \A(S)$ we know that for each subset $\{j_1,j_2\}$ in the partition $x_{j_1}+x_{j_2}+1> x_{i}=3$. Hence at least one of $x_{j_1},x_{j_2}$ is at least $2$. The number of subsets $\{j_1,j_2\}$ in the partition is at least $\lceil\frac{m-2}{2}\rceil$. Therefore,
\[
g-(m-1)=\sum_{j=1}^{m-1}(x_j-1)\geq 2+\left\lceil\frac{m-2}{2}\right\rceil \geq 2+\frac{m-2}{2}.\qedhere
\]
\end{proof}

We are now ready to prove Theorem \ref{thm:bijection}.  We prove it in two parts.
\begin{proposition}\label{kunz tuple necessary cond}
Fix integers $k_2\geq k_1\geq -1$. Suppose $g\geq 3k_1+2$, and $S\in\SS_g$ satisfies $m(S)=g-k_1$ and $e(S)=g-k_2$. Let $m=m(S)$.
Suppose $\KV_m(S) = (x_1,\dots,x_{m-1})$ and let $\overline{x}=(x_1,\dots,x_{2k_1+1})$. Then the following hold:
\begin{enumerate}
    \item $\{x_1,\dots,x_{m-1}\}\subseteq \{1,2,3\}$.
    \item If $i\geq 2k_1$, then $x_i\in\{1,2\}$.
    \item Whenever $i_1, i_2, i_3\in [1,2k_1+1]$ satisfy $i_1+i_2=i_3$, we have $(x_{i_1},x_{i_2},x_{i_3})\neq(1,1,3)$.
    \item $\#\{i\in [2k_1+2,m-1]\mid x_i=2\}=k_1+1-a(\overline{x})-2b(\overline{x}).$
    \item $a(\overline{x})+b(\overline{x})-c(\overline{x})=2k_1+1-k_2$.
    \item $e(S) = g-2k_1-1+a(\overline{x})+b(\overline{x})-c(\overline{x})$.
    \item $k_2\leq 2k_1+1$. 
\end{enumerate}
\end{proposition}
\noindent We note that verifying property (6) proves Proposition \ref{embedding_overlinex}.
\begin{proof}
We know that $m=g-k_1\geq 2k_1+2$. Next,
\[
3m+2-2g=3(g-k_1)+2-2g=g-(3k_1+2)+4\geq 4.
\]
This means that $2g<3m+2$. Lemma \ref{ki in 1,2,3} implies $\{x_1,\dots,x_{m-1}\}\subseteq \{1,2,3\}$.

Suppose $i$ satisfies $2k_1\leq i\leq m-1$. Assume for the sake of contradiction that $x_i=3$. By Lemma \ref{ki=3} we have $g\geq m+1+\left\lceil\frac{i-1}{2}\right\rceil$. 
However, 
\[
g\geq m+1+\left\lceil\frac{i-1}{2}\right\rceil \geq m+1+\left\lceil\frac{2k_1-1}{2}\right\rceil=m+1+k_1=g+1.
\]
This is a contradiction and we conclude that $x_i\in \{1,2\}$.

Suppose $i_1, i_2, i_3\in [1,2k_1+1]$ satisfy $i_1+i_2=i_3$. Suppose $x_{i_1}=x_{i_2}=1$. Since $(x_1,\dots,x_{m-1})$ is the Kunz coordinate vector of a numerical semigroup, we know that $x_{i_3}\leq x_{i_1}+x_{i_2}=2$.

Next, note that
\[
k_1+1=g-(m-1)=\sum_{i=1}^{m-1}(x_i-1) =a(\overline{x})+2b(\overline{x})+\#\{i\in [2k_1+2,m-1]\mid x_i=2\}.
\]
Next, we claim that $\A(S)$ is given by $m$ together with the elements $mx_i + i$ satisfying either
\begin{enumerate}
    \item $x_i = 1$, or 
    \item $x_i = 2$ where $i \in [1,2k_1+1]$ and there does not exist a $j$ satisfying $1 \le j < i$ with $x_j = x_{i-j} = 1$.
\end{enumerate}
It is clear that all these elements are elements of $\A(S)$. Suppose $a=mx_i+i$ is some other element of $\Ap(S;m)$. Then one of the following must hold:
\begin{itemize}[leftmargin=*]
    \item Case 1: $x_i=2$ where $i \in [1,2k_1+1]$ and there does exist a $j$ satisfying $1 \le j < i$ with $x_j = x_{i-j} = 1$. Then $x_i=x_{j}+x_{i-j}$, so $a \not\in \A(S)$.
    \item Case 2: $x_i=2$ where $i \in [2k_1+2,m-1]$.
    Assume for the sake of contradiction that $a \in \A(S)$. Lemma \ref{ki=2 is gen} implies that $g\geq m+\left\lceil\frac{i-1}{2}\right\rceil$. This implies that
    \[
    g\geq m+\left\lceil\frac{2k_1+1}{2}\right\rceil=m+k_1+1=g+1.
    \]
    This is a contradiction. Therefore $a\notin\A(S)$.
    \item Case 3: $x_i=3$. Assume for the sake of contradiction that $a \in \A(S)$. By Lemma \ref{ki=3 is gen}, $g\geq \frac{3m}{2}$. This implies that
    \[
    k_1=g-m\geq \frac{m}{2}=\frac{g-k_1}{2}\geq \frac{2k_1+2}{2}=k_1+1,
    \]
    which is a contradiction. Therefore $a\notin\A(S)$.
\end{itemize}
This characterization of $\A(S)$ implies that 
\begin{eqnarray*}
    g-k_2& =& e(S)\\
    & = & 1+\big(2k_1+1-a(\overline{x})-b(\overline{x})\big)+\#\{i\in[2k_1+2,m-1]\mid x_i=1\}+\big(a(\overline{x})-c(\overline{x})\big)\\
    &=& 2k_1+2-b(\overline{x})-c(\overline{x})+\big(m-1-(2k_1+1)-\#\{i\in [2k_1+2,m-1]\mid x_i=2\}\big)\\
    &=& -b(\overline{x})-c(\overline{x})+(g-k_1)-\big(k_1+1-a(\overline{x})-2b(\overline{x})\big)\\
    &=& g-2k_1-1+a(\overline{x})+b(\overline{x})-c(\overline{x}).
\end{eqnarray*}
We conclude that $a(\overline{x})+b(\overline{x})-c(\overline{x})=2k_1+1-k_2$.

Finally since $c(\overline{x})\leq a(\overline{x})$, we see that $a(\overline{x})+b(\overline{x})-c(\overline{x})\geq 0$ and so $k_2\leq 2k_1+1$.
\end{proof}

We now prove the other direction in Theorem \ref{thm:bijection}.
\begin{proposition}
Fix integers $-1\leq k_1\leq k_2\leq 2k_1+1$. Also fix $m\geq 2k_1+2$ and a tuple of positive integers  $(x_1,x_2,\dots,x_{m-1})$. Let $\overline{x}=(x_1,\dots,x_{2k_1+1})$. Suppose we have the following:
\begin{enumerate}
    \item $\{x_1,\dots,x_{m-1}\}\subseteq \{1,2,3\}$.
    \item If $i\geq 2k_1+2$ then $x_i\in\{1,2\}$.
    \item Whenever $i_1, i_2, i_3\in [1,2k_1+1]$ satisfy $i_1+i_2=i_3$, we have $(x_{i_1},x_{i_2},x_{i_3})\neq(1,1,3)$.
    \item $\#\{i\in [2k_1+2,m-1]\mid x_i=2\}=k_1+1-a(\overline{x})-2b(\overline{x}).$
    \item $a(\overline{x})+b(\overline{x})-c(\overline{x})=2k_1+1-k_2$.
\end{enumerate}
Let
\[
S=m\N_0\cup \bigcup_{i=1}^{m-1}(i+mx_i+m\N_0).
\]
Then $S$ is a numerical semigroup satisfying $m(S)=m,\ g(S)=m+k_1$, and $e(S)=g(S)-k_2$.
\end{proposition}
\begin{proof}
Theorem \ref{Kunz_bijection} implies that if $(x_1,\ldots, x_{m-1})$ satisfies the first three conditions, then it is the Kunz coordinate vector of a numerical semigroups of multiplicity $m$. Notice that
\begin{equation*}
    \begin{split}
    g(S)-(m-1)&=\sum_{i=1}^{m-1}(x_i-1) =\#\{i\in [1,m-1]\mid x_i=2\}+2\#\{i\in [1,m-1]\mid x_i=3\}\\
    &=a(\overline{x})+\#\{i\in [2k_1+2,m-1]\mid x_i=2\}+2b(\overline{x})\\
    &=a(\overline{x})+2b(\overline{x})+\big(k_1+1-a(\overline{x})-2b(\overline{x})\big)\\
    &=k_1+1.
    \end{split}
\end{equation*}
This means that $g(S)=m+k_1$.
Suppose $e(S)=g(S)-k$.  It is clear that $k_1\leq k$, and also that $g(S)\geq 3k_1+2$. Therefore, by Proposition \ref{kunz tuple necessary cond}(5) we see that $a(\overline{x})+b(\overline{x})-c(\overline{x})=2k_1+1-k$. This implies $k=k_1$.
\end{proof}

We end this paper by including some data related to the sets $\Y(k)$ and the polynomials $H_l(x)$ and $f_k(x)$ of this section.  The initial $\Y(k)$ are as follows:
\[
\Y(-1)=\{\emptyset\},\;\;\;
\Y(0)=\{(1),(2)\},
\]
\begin{equation*}
    \begin{split}
    \Y(1)=\{&(1,1,1),(1,1,2),(1,2,1),(2,1,1),(2,2,1),(2,1,2),(1,2,2),(3,1,1) \},
    \end{split}
\end{equation*}
\begin{equation*}
    \begin{split}
    \Y(2)=\{&(1,1,1,1,1),(2,1,1,1,1),(1,2,1,1,1),(1,1,2,1,1),(1,1,1,2,1),(1,1,1,1,2),\\
    &(2,2,1,1,1),(2,1,2,1,1),(2,1,1,2,1),(2,1,1,1,2),(1,2,2,1,1),(1,2,1,2,1),\\
    &(1,2,1,1,2),(1,1,2,2,1),(1,1,2,1,2),(1,1,1,2,2),(2,2,2,1,1),(2,2,1,2,1),\\
    &(2,2,1,1,2),(2,1,2,2,1),(2,1,2,1,2),(2,1,1,2,2),(1,2,2,2,1),(1,2,2,1,2),\\
    &(1,2,1,2,2),(1,1,2,2,2),(3,1,1,1,1),(3,2,1,1,1),(3,1,2,1,1),(3,1,1,2,1),\\
    &(3,1,1,1,2),(2,3,1,1,1),(2,1,3,1,1),(1,2,3,1,1)\}.
    \end{split}
\end{equation*}

The first few polynomials are as follows:
$$H_{-1}(t)=1,\;\;\;\;\; H_{0}(t)=1,$$
$$H_{1}(t)=t,\;\;\;\;\ \ \ \ \ \ \ H_{2}(t)=t+1,$$
$$H_{3}(t)=\frac{t^2}{2}-\frac{3t}{2}+2,\;\;\;\;
H_{4}(t)=\frac{t^2}{2}-\frac{t}{2}-2.$$

\section*{Acknowledgments}
The first author was supported by NSF Grants DMS 1802281 and DMS 2154223.  The authors thank Shalom Eliahou, Daniel Zhu, and Sean Li for helpful comments.  The authors thank the organizers of the INdaM Workshop: \emph{International Meeting on Numerical Semigroups 2022} where the idea for this paper was developed.

\end{document}